\documentclass{amsart}
\usepackage[a4paper,hmargin={2.5cm,2.5cm},vmargin={2.5cm,2.5cm},heightrounded,marginparwidth=2.2cm,marginparsep=0.1cm]{geometry}

\usepackage[utf8]{inputenc}
\usepackage[T1]{fontenc}
\usepackage{lmodern}

\usepackage[leqno]{amsmath}
\usepackage{amssymb,amsthm}
\usepackage[mathscr]{euscript}
\usepackage{bbm}
\usepackage{dsfont}
\usepackage[shortlabels]{enumitem}
\usepackage[all,2cell]{xy} \UseTwocells
\usepackage{mathtools}

\theoremstyle{plain}
\newtheorem{theorem}{Theorem}[section]
\newtheorem{lemma}[theorem]{Lemma}
\newtheorem{proposition}[theorem]{Proposition}
\newtheorem{corollary}[theorem]{Corollary}

\theoremstyle{definition}
\newtheorem{definition}[theorem]{Definition}
\newtheorem{examples}[theorem]{Examples}
\newtheorem{example}[theorem]{Example}

\theoremstyle{remark}
\newtheorem{remark}[theorem]{Remark}

\newlist{tfae}{enumerate}{1}
\setlist[tfae,1]{label=(\roman*)}

\usepackage{chngcntr} \counterwithin*{equation}{section}

\makeatletter \def\nlabel#1#2{\begingroup #2%
  \def\@currentlabel{#2}%
  \phantomsection\label{#1}\endgroup } \makeatother
\newcommand{\litem}[2]{\item[\nlabel{#1}{#2}]}
\setlist[description]{font=\normalfont\space,labelindent=\parindent}

\newcommand{\ff}{\mathfrak{f}}

\newcommand{\fj}{\mathfrak{j}}

\newcommand{\fr}{\mathfrak{r}}

\newcommand{\fv}{\mathfrak{v}}
\newcommand{\fw}{\mathfrak{w}}
\newcommand{\fx}{\mathfrak{x}}
\newcommand{\fy}{\mathfrak{y}}
\newcommand{\fz}{\mathfrak{z}}

\newcommand{\calA}{\mathcal{A}}

\newcommand{\fW}{\mathfrak{W}}
\newcommand{\fX}{\mathfrak{X}}

\DeclareMathOperator{\Pdn}{P_{\!_\downarrow}\!}

\DeclareMathOperator{\Dnw}{\Downarrow\!}

\DeclareMathOperator{\SOB}{B_S}

\DeclareSymbolFont{symbolsA}{U}{txsya}{m}{n}
\DeclareSymbolFont{symbolsC}{U}{txsyc}{m}{n}
\DeclareMathSymbol{\multimapdot}{\mathrel}{symbolsC}{20}
\DeclareMathSymbol{\multimapdotinv}{\mathrel}{symbolsC}{21}
\DeclareMathSymbol{\multimap}{\mathrel}{symbolsA}{40}
\DeclareMathSymbol{\multimapinv}{\mathrel}{symbolsC}{18}

\newcommand{\blackright}{\multimapdot}
\newcommand{\blackleft}{\multimapdotinv}

\newcommand{\catfont}[1]{\mathsf{#1}}

\newcommand{\catX}{\catfont{X}}
\newcommand{\catA}{\catfont{A}}

\newcommand{\ddf}{\mathcal{D}}

\newcommand{\SET}{\catfont{Set}}
\newcommand{\REL}{\catfont{Rel}}
\newcommand{\ORD}{\catfont{Ord}}
\newcommand{\MET}{\catfont{Met}}
\newcommand{\UMET}{\catfont{UMet}}
\newcommand{\BMET}{\catfont{BMet}}
\newcommand{\PROBMET}{\catfont{ProbMet}}
\newcommand{\TOP}{\catfont{Top}}
\newcommand{\APP}{\catfont{App}}
\newcommand{\PROBAPP}{\catfont{ProbApp}}

\newcommand{\BOOLE}{\catfont{Boole}}

\newcommand{\COMPHAUS}{\catfont{CompHaus}}
\newcommand{\ORDCH}{\catfont{OrdCH}}
\newcommand{\METCH}{\catfont{MetCH}}
\newcommand{\PROBMETCH}{\catfont{ProbMetCH}}

\newcommand{\STCOMP}{\catfont{StablyComp}}
\newcommand{\POSCH}{\catfont{PosComp}}

\newcommand{\Rels}[1]{#1\text{-}\catfont{Rel}}
\newcommand{\Cats}[1]{#1\text{-}\catfont{Cat}}

\newcommand{\Dists}[1]{#1\text{-}\catfont{Dist}}

\newcommand{\two}{\mathbf{2}}
\newcommand{\V}{\mathcal{V}}
\newcommand{\quantale}{(\V,\otimes,k)}
\newcommand{\Pp}{\overleftarrow{[0,\infty]}_+}
\newcommand{\Pm}{\overleftarrow{[0,\infty]}_\wedge}

\newcommand{\op}{\mathrm{op}}
\newcommand{\sep}{\mathrm{sep}}

\newcommand{\comp}{\mathrm{comp}}

\newcommand{\kleisli}{\circ}

\newcommand{\relto}{\mathrel{\mathmakebox[\widthof{$\xrightarrow{\rule{1.45ex}{0ex}}$}]
    {\xrightarrow{\rule{1.45ex}{0ex}}\hspace*{-2.4ex}{\mapstochar}\hspace*{1.8ex}}}}
\newcommand{\modto}{\mathrel{\mathmakebox[\widthof{$\xrightarrow{\rule{1.45ex}{0ex}}$}]
    {\xrightarrow{\rule{1.45ex}{0ex}}\hspace*{-2.8ex}{\circ}\hspace*{1ex}}}}

\newcommand{\krelto}{\mathrel{\mathmakebox[\widthof{$\xrightarrow{\rule{1.45ex}{0ex}}$}]
    {\xrightharpoonup{\rule{1.45ex}{0ex}}\hspace*{-2.4ex}{\mapstochar}\hspace*{1.8ex}}}}
\newcommand{\kmodto}{\mathrel{\mathmakebox[\widthof{$\xrightarrow{\rule{1.45ex}{0ex}}$}]
    {\xrightharpoonup{\rule{1.45ex}{0ex}}\hspace*{-2.8ex}{\circ}\hspace*{1ex}}}}

\makeatletter

\def\slashedarrowfill@#1#2#3#4#5{$\m@th\thickmuskip0mu\medmuskip\thickmuskip\thinmuskip\thickmuskip\relax#5#1\mkern-7mu
  \cleaders\hbox{$#5\mkern-2mu#2\mkern-2mu$}\hfill \mathclap{#3}\mathclap{#2}
  \cleaders\hbox{$#5\mkern-2mu#2\mkern-2mu$}\hfill \mkern-7mu#4$}

\newcommand*{\rightrelarrowfill@}{\slashedarrowfill@\relbar\relbar{\raisebox{0pc}{$\mapstochar$}}\rightarrow}
\newcommand*{\xrelto}[2][]{\ext@arrow 0055{\rightrelarrowfill@}{\;#1\;}{\;#2\;}}
  
\newcommand*{\rightmodarrowfill@}{\slashedarrowfill@\relbar\relbar{\raisebox{0pc}{$\hspace{0.3em}\circ$}}\rightarrow}
\newcommand*{\xmodto}[2][]{\ext@arrow 0055{\rightmodarrowfill@}{\;#1\;}{\;#2\;}}
  
\newcommand*{\rightkrelarrowfill@}{\slashedarrowfill@\relbar\relbar{\raisebox{0pc}{$\hspace{0.3em}\mapstochar$}}\rightharpoonup}
\newcommand*{\xkrelto}[2][]{\ext@arrow
  0055{\rightkrelarrowfill@}{\;#1\;}{\;#2\;}}
  
\newcommand*{\rightkmodarrowfill@}{\slashedarrowfill@\relbar\relbar{\raisebox{0pc}{$\hspace{0.3em}\circ$}}\rightharpoonup}
\newcommand*{\xkmodto}[2][]{\ext@arrow
  0055{\rightkmodarrowfill@}{\;#1\;}{\;#2\;}}

\makeatother

\newcommand\adjunct[2]{\xymatrix@=8ex{\ar@{}[r]|{\top}\ar@<1mm>@/^2mm/[r]^{{#2}}
    & \ar@<1mm>@/^2mm/[l]^{{#1}}}}
\newcommand\adjunctop[2]{\xymatrix@=8ex{\ar@{}[r]|{\bot}\ar@<1mm>@/^2mm/[r]^{{#2}}
    & \ar@<1mm>@/^2mm/[l]^{{#1}}}}

\newcommand\adjunctrel[2]{\xymatrix@=8ex{\ar@{}[r]|{\top}\ar@<1mm>@/^2mm/[r]|-{\object@{|}}^{{#2}}
    & \ar@<1mm>@/^2mm/[l]|-{\object@{|}}^{{#1}}}}
\newcommand\adjunctrelop[2]{\xymatrix@=8ex{\ar@{}[r]|{\bot}\ar@<1mm>@/^2mm/[r]|-{\object@{|}}^{{#2}}
    & \ar@<1mm>@/^2mm/[l]|-{\object@{|}}^{{#1}}}}

\newcommand\adjunctmod[2]{\xymatrix@=8ex{\ar@{}[r]|{\top}\ar@<1mm>@/^2mm/[r]|-{\object@{o}}^{{#2}}
    & \ar@<1mm>@/^2mm/[l]|-{\object@{o}}^{{#1}}}}
\newcommand\adjunctmodop[2]{\xymatrix@=8ex{\ar@{}[r]|{\bot}\ar@<1mm>@/^2mm/[r]|-{\object@{o}}^{{#2}}
    & \ar@<1mm>@/^2mm/[l]|-{\object@{o}}^{{#1}}}}

\newcommand\adjunctkmod[2]{\xymatrix@=8ex{\ar@{}[r]|{\top}\ar@{-^>}@<1mm>@/^2mm/[r]|-{\object@{o}}^{{#2}}
    & \ar@{-^>}@<1mm>@/^2mm/[l]|-{\object@{o}}^{{#1}}}}
\newcommand\adjunctkmodop[2]{\xymatrix@=8ex{\ar@{}[r]|{\bot}\ar@{-^>}@<1mm>@/^2mm/[r]|-{\object@{o}}^{{#2}}
    & \ar@{-^>}@<1mm>@/^2mm/[l]|-{\object@{o}}^{{#1}}}}

\newcommand{\Uxi}{U_{\!_\xi}}
\newcommand{\Uxione}{U_{\!_{\xi_1}}}
\newcommand{\Uxitwo}{U_{\!_{\xi_2}}}

\newcommand{\monadfont}[1]{\mathbbm{#1}}
\newcommand{\mT}{\monadfont{T}}
\newcommand{\mU}{\monadfont{U}}

\newcommand{\monad}{(T,m,e)}
\newcommand{\umonad}{(U,m,e)}

\newcommand{\theoryfont}[1]{\mathscr{#1}}

\newcommand{\thU}{\theoryfont{U}}

\newcommand{\utheory}{(\mU,\V,\xi)}
\newcommand{\utheoryone}{(\mU,\V_1,\xi_1)}
\newcommand{\utheorytwo}{(\mU,\V_2,\xi_2)}

\newcommand{\FgtSet}[1]{O_{#1}}
\newcommand{\FgtOrd}[1]{\widetilde{O}_{#1}}

\DeclareMathAlphabet{\mathpzc}{OT1}{pzc}{m}{it}

\newcommand{\coyoneda}{\mathpzc{h}}

\newcommand{\DirLimCls}{\Phi_{\!\Delta}}

\newcommand{\Fne}{f_{n,\varepsilon}}

\DeclareMathOperator{\upc}{\uparrow\!}
\DeclareMathOperator{\downc}{\downarrow\!}

\DeclareMathOperator{\Cauchy}{Cauchy}
\DeclareMathOperator{\spec}{spec}

\newcommand{\doo}[1]{\overset{\centerdot}{#1}}

\newcommand{\Tast}{\circledast}

\newcommand{\df}[1]{\emph{\textbf{#1}}}

\newcommand{\field}[1]{\mathds{#1}}

\newcommand{\N}{\field{N}}

\newcommand{\zerone}{[0,1]}
\newcommand{\zerinf}{[0,\infty]}

\title{Convergence and quantale-enriched categories}

\author{Dirk Hofmann}

\address{Center for Research and Development in Mathematics and Applications,
  Department of Mathematics, University of Aveiro, 3810-193 Aveiro, Portugal}

\email{dirk@ua.pt}

\author{Carla Reis}

\address{Polytechnic Institute of Coimbra, College of Management and Technology
  of Oliveira do Hospital, 3400-124 Oliveira do Hospital, Portugal \and CIDMA,
  University of Aveiro, Portugal}

\email{carla.reis@estgoh.ipc.pt}

\thanks{Partial financial assistance by Portuguese funds through CIDMA (Center
  for Research and Development in Mathematics and Applications), and the
  Portuguese Foundation for Science and Technology (``FCT -- Funda\c{c}\~ao para
  a Ci\^encia e a Tecnologia''), within the project UID/MAT/04106/2013 is
  gratefully acknowledged.}

\date{\today}

\usepackage[backref=page,hypertexnames=false]{hyperref} 

\begin{document}

\begin{abstract}
  Generalising Nachbin's theory of ``topology and order'', in this paper we
  continue the study of quantale-enriched categories equipped with a compact
  Hausdorff topology. We compare these $\V$-categorical compact Hausdorff spaces
  with ultrafilter-quantale-enriched categories, and show that the presence of a
  compact Hausdorff topology guarantees Cauchy completeness and (suitably
  defined) codirected completeness of the underlying quantale enriched category.
\end{abstract}



\maketitle

\section{Introduction}
\label{sec:introduction}

\subsection{Motivation}
\label{sec:motivation}

This paper continuous a line of research initiated in \cite{Tho09} which
combines Nachbin's theory of ``topology and order'' \cite{Nac50} with the
setting of monad-quantale enriched categories \cite{HST14}. Over the past
century, the combination of order structures with compact Hausdorff topologies
has proven to be very fruitful in various parts of mathematics: in the form of
spectral spaces they appear in Stone duality for distributive lattices
\cite{Sto38} and Hochster's characterisation of prime spectra of commutative
rings \cite{Hoc69}, the connection between spectral spaces and certain partially
ordered compact spaces was made explicit in \cite{Pri70,Pri72} (see also
\cite{Cor75,Fle00}), and was further extended to an equivalence between all
partially ordered compact spaces and stably compact topological spaces in the
1970's (see \cite{GHK+80}). Subsequently, stably compact spaces have also played
a central role in the development of domain theory, see \cite{Law11} for
details. In a more general context, compact Hausdorff spaces combined with the
structure of a quantale-enriched category have been essential in study of
topological structures as categories: they appear in the definition of ``dual
space'', still implicitly in \cite{CH09} and more explicitly in
\cite{Hof13,GH13,Hof14}. This notion turned out to be an essential ingredient in
the study of (co)completeness properties of monad-quantale enriched
categories. In \cite{CCH15} we also explain the connection of Nachbin's work
with the theory of multicategories \cite{Her00,Her01}.

Motivated by this development, we focus here on the ultrafilter monad and study
quantale-enriched categories equipped with a compact Hausdorff topology; our
examples include ordered, metric, and probabilistic metric compact Hausdorff
spaces. We show that the presence of a compact Hausdorff topology guarantees
Cauchy completeness and (suitably defined) codirected completeness of the
underlying quantale enriched category. Our investigation relies on a connection
between these $\V$-categorical compact Hausdorff spaces and monad-quantale
enriched categories which generalises the equivalence between partially ordered
compact spaces and stably compact topological spaces
(see~\cite[Section~III.5]{HST14}). Another important ingredient is the concept
of Cauchy completeness \emph{{\`a} la} Lawvere for monad-quantale enriched
categories as introduced in \cite{CH09}. In order to include probabilistic
metric spaces in our study, our setting is slightly weaker than the one
considered in \cite{CH09}. Due to this weaker assumptions, we have to overcome
some technical difficulties which force us to revise and extend some notions and
results of \cite{CH09}.

In order to explain our motivation more in detail, we find it useful to place it
in a historical context.

\subsection{Historical background}
\label{sec:hist-backgr}

Right from its origins at the beginning of the 20\textsuperscript{th} century,
one major concern of set-theoretic topology was the development of a
satisfactory notion of convergence. This in turn was motivated by the increasing
use of abstract objects in mathematics: besides numbers, mathematical theories
deal with sequences of functions, curves, surfaces,\dots. To the best of our
knowledge, a first attempt to treat convergence abstractly is presented in
\cite{Fre06}. Whereby the main contribution of \cite{Fre06} is the concept of a
(nowadays called) metric space, the starting point of \cite{Fre06} is actually
an abstract theory of sequential convergence. Fr\'{e}chet considers a function
associating to every sequence of a set $X$ a point of $X$, its convergence
point, subject to the following axioms:
\begin{description}
  \litem{d:item1}{(A)} Every constant sequence $(x,x,\dots)$ converges to $x$;
  \litem{d:item2}{(B)} If a sequence $(x_n)_{n\in\N}$ converges to $x$, then
  also every subsequence of $(x_n)_{n\in\N}$ converges to $x$.
\end{description}
Under these conditions, Fr\'{e}chet gave indeed a generalisation of
Weierstra\ss's theorem \cite{Fre04}; however, these constrains seem to be to
weak in general \emph{since the limes axioms \ref{d:item1}, \ref{d:item2} have
  still very little power \dots} (\cite[page~266]{Hau65}, original in German;
our translation).
In \cite{Hau14}, Hausdorff introduces the notion of topological space via
neighbourhood systems
and compares the notions of distance, topology and sequential convergence as
\emph{\dots the theory of distances seems to be the most specific, the limes
  theory the most general\dots} (\cite[page~211]{Hau65}, original in German; our
translation).
In the introduction to Chapter 7 ``Punktmengen in allgemeinen R\"aumen'',
Hausdorff affirms that \emph{the most beautiful triumph of set theory lies in
  its application to the point sets of the space, in the clarification and
  sharpening of the geometric notions\dots} (\cite[page~209]{Hau65}, original in
German; our translation).
According to Hausdorff, these geometric notions not only involve approximation
and distance, but also the theory of (partially) ordered sets to which he
dedicates a substantial part of his book. Thinking of an order relation on a set
$M$ as a function
\[
  f:M\times M \longrightarrow \{<,>,=\},
\]
Hausdorff also foresees that (\cite[page~210]{Hau65}, original in German; our
translation)
\begin{quote}
  \emph{Now there is nothing in the way of a generalisation of this idea, and we
    can think of an arbitrary function of pairs of points which associates to
    each pair $(a,b)$ of elements of a set $M$ a specific element $n=f(a, b)$ of
    a second set $N$. Generalising further, we can consider a function of
    triples, sequences, complexes, subsets, etc.}
\end{quote}
In particular, Hausdorff already presents metric spaces as a direct
generalisation of ordered sets where now $f$ associates to each pair $(a,b)$ the
distance between $a$ and $b$. This point of view was taken much further in
\cite{Law73}: not only the structure but also the axioms of an ordered set and
of a metric space are very similar and, moreover, can be seen as instances of
the definition of a category. Furthermore, Hausdorff sees also the definition of
a topological space as a generalisation of the concept of a partially ordered
set: instead of a relation between points, sequential convergence relates
sequences with their convergence points, and a neighbourhood system relies on a
relation between points and subsets. Surprisingly, also here the relevant axioms
on such relations can be formulated so that they resemble the ones of a
partially ordered set. We refer the reader to the monograph \cite{HST14} for an
extensive presentation of this theory and for further pointers to the
literature.

Clearly, Hausdorff considers topologies as generalised partial orders; however,
a more direct relation between the two concepts was only given more than twenty
years later. In \cite{Ale37}, Alexandroff observes that every partial order on a
set $X$ defines a topology, and from this topology one can reconstruct the given
order relation via
\begin{equation}\label{d:eq:1}
  x\le y\iff \text{$y$ belongs to every neighbourhood of $x$}.
\end{equation}
Furthermore, Alexandroff characterises the topological spaces obtained this way
as the so called ``diskrete R\"aume'', namely as those T0 spaces where the
intersection of open subsets is open. These spaces, without assuming the T0
separation axiom, are nowadays called \df{Alexandroff spaces}. In this paper we
depart from Hausdorff's nomenclature since partial orders seem to be more
frequent than total ones. Therefore we call a binary relation $\leq$ on a set
$X$ an \df{order relation} whenever $\leq$ is reflexive and transitive, and
speak of a \df{total order} whenever all elements are comparable. Furthermore,
we think of the anti-symmetry condition as a(n often unnecessary) separation
axiom. We write $\ORD$ for the category of ordered sets and monotone maps and,
with $\TOP$ denoting the category of topological spaces and continuous maps,
Alexandroff's construction extends to a functor
\[
  \TOP \longrightarrow\ORD
\]
which commutes with the underlying forgetful functors to the category $\SET$ of
sets and functions. The order relation defined by \eqref{d:eq:1} is now known as
the \df{specialisation order} of the space $X$. This order looses most of the
topological information of a space $X$ and does not seem to be very useful for
the study of topological properties. Nevertheless, there are some properties of
a space $X$ which are reflected in the specialisation order, in particular the
lower separation axioms:
\begin{itemize}
\item $X$ is T0 if and only if the specialisation order of $X$ is separated
  (=anti-symmetric); and
\item $X$ is T1 if and only if the specialisation order of $X$ is discrete.
\end{itemize}
The latter equivalence might be the reason why this order relation does not play
a dominant role in general topology. More interesting seems to be the reverse
question: which order properties are guaranteed by certain topological
properties? For instance, the following observation is very relevant for our
paper:
\begin{itemize}
\item if $X$ is sober, then the specialisation order of $X$ is directed complete
  (see \cite[Lemma~II.1.9]{Joh86}).
\end{itemize}

The specialisation order plays also a role in Hochster's study of ring spectra:
\cite{Hoc69} characterises the prime spectra of commutative rings as precisely
Stone's spectral spaces \cite{Sto38}. Here, for a commutative ring $R$, the
order of the topology on $\spec(R)$ should match the inclusion order of prime
ideals; by that reason Hochster considers the dual of the specialisation
order. Motivated by the convergence theoretic approach described below, in this
paper we will also consider this \emph{underlying} order of a topological space
instead of the specialisation order. A deep connection between topological
properties and order properties is made in \cite{Sco72} where injective
topological T0 spaces are characterised in terms of their underlying partial
order.

Whereby in the considerations above the order relation is the one induced by a
given topology, a different road was taken in \cite{Nac50} in his study of
ordered topological spaces where topology and order are two independent
structures, subject to a mild compatibility condition. This combination allows
for an substantial extension of the scope of various important notions and
results in topology, we mention here the concept of order-normality and the
Urysohn Lemma. Of special interest to us is a particular class of separated
ordered topological spaces, namely the compact ones, which are described in
\cite{Jun04} as ``precisely the T0 analogues of compact Hausdorff
spaces''. These spaces can be equivalently described in purely topological
terms: firstly, there is a comparison functor
\[
  K:\POSCH \longrightarrow\TOP
\]
between the category $\POSCH$ of separated ordered compact spaces and monotone
continuous maps and $\TOP$; secondly, this functor restricts to an equivalence
$\POSCH\simeq\STCOMP$ where $\STCOMP$ denotes the category of stably compact
spaces and spectral maps. These facts are known since the beginning of the
1970's and were first published in \cite{GHK+80}. To explain this connection
better, we find it useful to return to the story of convergence.

After Hausdorff's fundamental book \cite{Hau65}, the notion of convergence does
not seem to have played a prominent role in the development of topology. The
notion of sequence proved to be insufficient, and only in the 1930's
\cite{Bir37} appeared a characterisation of topological T1 spaces in terms of an
abstract concept of convergence based on the notion of \emph{Moore-Smith
  sequence} \cite{MS22,Moo15}. At the same time, Cartan introduced the concept
of filter convergence \cite{Car37,Car37a}, and this idea was met with enthusiasm
within the Bourbaki group \cite{Bou42}. However, it seems to us that this
enthusiasm was not shared by most treatments of topology as convergence plays
often only a secondary role. We refer to \cite{Cla13} for more information on
convergence and its history.

Using either filters or nets (as Moore-Smith sequence are typically called
nowadays), convergence finally conquered its appropriate place in topology. This
also led to the consideration of abstract (ultra)filter convergence structures,
we mention here the papers \cite{Gri60,Gri61,CF67} where topological convergence
structures are characterised among more general ones. In our opinion, the most
useful descriptions were obtained around 1970: firstly, Manes characterises
compact Hausdorff spaces as precisely the Eilenberg--Moore algebras for the
ultrafilter monad $\mU=\umonad$ on $\SET$ \cite{Man69}, and Barr characterises
topological spaces as the \emph{lax} algebras for the ultrafilter monad
\cite{Bar70}. More in detail, a compact Hausdorff space is given by a set $X$
together with a \emph{map} $\alpha:UX\to X$ so that the diagrams
\begin{align*}
  \xymatrix{X\ar[r]^{e_X}\ar[dr]_{1_X} & UX\ar[d]^{\alpha} \\   & X}
                                                                &&\text{and}&&
                                                                               \xymatrix{UUX\ar[r]^{m_X}\ar[d]_{U\alpha} & UX\ar[d]^{\alpha} \\ UX\ar[r]_{\alpha} & X}
\end{align*}
commute in $\SET$; whereby a general topological space is given by a set $X$
together with a \emph{relation} $a:UX\relto X$ so that the inequalities
\begin{align*}
  \xymatrix{X\ar[r]|-{\object@{|}}^{e_X}\ar[dr]|-{\object@{|}}_{1_X}^\le & UX\ar[d]|-{\object@{|}}^{a} \\  & X}
                                                                                                           &&\text{and}&&
                                                                                                                          \xymatrix{UUX\ar[r]|-{\object@{|}}^{m_x}\ar[d]|-{\object@{|}}_{Ua}\ar@{}[dr]|\le & UX\ar[d]|-{\object@{|}}^{a} \\ UX\ar[r]|-{\object@{|}}_{a} & X}
\end{align*}
hold in the ordered category $\REL$ of sets and relations. Elementwise, the
latter axioms read as
\begin{align*}
  e_X(x)\to x &&\text{and}&& (\fX\to\fx\;\&\;\fx\to x)\implies m_X(\fX)\to x,
\end{align*}
for all $x\in X$, $\fx\in UX$ and $\fX\in UUX$. Note that the second condition
talks about the convergence of an ultrafilter of ultrafilters $\fX$ to an
ultrafilter $\fx$, which comes from applying the ultrafilter functor $U$ to the
\emph{relation} $a:UX\relto X$. Hence, these description involves the additional
difficulty of extending the functor $U:\SET\to\SET$ in a suitable way to a
locally monotone endofunctor on $\REL$; but it is extremely useful since it does
not only provide axioms but also a calculus to deal with these axioms since they
are formulated within the structure of the ordered category $\REL$. Barr's
characterisation gives also new evidence to Hausdorff's intuition that
topological spaces are generalised orders, as the two axioms are clearly
reminiscent to the reflexivity and the transitivity condition defining an order
relation. We also note that the underlying order of a topology $a:UX\relto X$ is
simply the composite $a\cdot e_X:X\relto X$.

Using this language, Tholen \cite{Tho09} shows that an ordered compact Hausdorff
space can be equivalently described as a set $X$ equipped with an order relation
$\le:X\relto X$ and a compact Hausdorff topology $\alpha:UX\to X$ which must be
compatible in the sense that
\[
  \alpha:(UX,U\leq) \longrightarrow (X,\leq)
\]
is monotone. Moreover, the object part of the functor $K:\POSCH\to\TOP$
mentioned above can now be simply described by relational composition
\[
  (X,\leq,\alpha) \longmapsto (X,\leq\cdot\alpha);
\]
a simple calculation shows that $\leq\cdot\alpha:UX\relto X$ satisfies indeed
the two axioms of a topology. More importantly, as already initiated in
\cite{Tho09}, this approach paves the way to mix topology with metric structures
or other ``generalised orders'' in the spirit of Hausdorff; or better: enriched
categories in the spirit of Lawvere \cite{Law73}. Undoubtedly, topology is
already omnipresent in the study of metric spaces; however, there does not seem
to exist a systematic account in the literature thinking of metric and topology
as a generalisation of Nachbin's ordered topological spaces. This motivation
brings us to the following considerations.
\begin{itemize}
\item Instead of analysing a metric space $(X,d)$ using the topology
  \emph{induced} by $d$, we ask what properties of $d$ are ensured by a compact
  Hausdorff topology \emph{compatible} with $d$.
\item To answer this question, we look back and ask the same question for the
  ordered case. Surprisingly, there is a quick answer: since every separated
  ordered compact space corresponds to a stably compact space which is in
  particular sober, every separated ordered compact space has codirected infima
  and, by duality, also directed suprema.
\item To transport this argumentation back to the metric case, we need a metric
  variant of sober topological spaces, which is provided by the notion of sober
  approach space \cite{Low97,BLO06,Olm05}.
\item we also consider the notion of codirected completeness for metric spaces
  which implies Cauchy completeness. We compare this notion to other concepts of
  (co)directedness in the literature.
\end{itemize}

The principal aim of this paper is to present a theory which encompasses the
steps above, for quantale-enriched categories equipped with a compact Hausdorff
topology; our examples include ordered, metric, and probabilistic metric compact
Hausdorff spaces. We place this study in the general framework of topological
theories \cite{Hof07} and monad-quantale-enriched categories (see \cite{HST14}),
for the ultrafilter monad $\mU$ on $\SET$.

\section{Basic notions}
\label{sec:basic-notions}

In this section we recall various aspects of the theory of quantale-enriched
categories. All results presented here are well-known, for more information
about enriched category theory we refer to the classic \cite{Kel82} and to
\cite{Stu14}. In our examples we focus on quantales based on the lattices
$\two=\{0,1\}$, $[0,1]$, $[0,\infty]$ and the lattice $\ddf$ of distribution
functions.

\subsection{Quantale}
A \df{quantale} $\V=\quantale$ is a complete lattice $\V$, with the order
relation denoted by $\leq$, equipped with a monoidal structure given by a
commutative binary operation $\otimes$, with identity $k$, which distributes
over joins:
\[
  u\otimes\left(\bigvee_{i\in I} u_i\right)=\bigvee_{i\in I}(u\otimes u_i).
\]
Thus, by Freyd's Adjoint Functor Theorem, for each $u\in \V$, the monotone map
$u\otimes -:\V\to \V$ has a right adjoint $\hom(u,-)$ characterised by
\[
  u\otimes v\leq w \iff v\leq \hom(u,w),
\]
for all $v,w\in\V$. Our principal examples include the following.
\begin{examples}\label{exs:1}
  \begin{enumerate}
  \item\label{item:1} The two element chain $\two=\{0,1\}$ of truth values with
    $0\leq 1$ is a quantale for $\otimes=\&$ being the logical operation
    ``and''; in this case $\hom(u,v)$ is just the implication $u\implies
    v$. More general, every Heyting algebra with $\otimes =\wedge$ being infimum
    and the identity given by the top element $\top$ is a quantale.
  \item\label{item:2} The extended real half line $\overleftarrow{[0,\infty]}$
    ordered by the ``greater or equal'' relation $\geqslant$ becomes a quantale
    with the tensor product given by the usual addition $+$, denoted by
    $\Pp$. In this case, $\hom(u,v)=u\ominus v=\max\{u-v,0\}$, for all
    $u,v\in [0,\infty]$. According to (\ref{item:1}), one can also equip
    $\overleftarrow{[0,\infty]}$ with the infimum $\otimes = \max$ of this
    lattice, we denote the resulting quantale as $\Pm$.
  \item\label{item:3} Similarly to (\ref{item:2}), we consider the unit interval
    $[0,1]$ with the ``greater or equal'' relation $\geqslant$ and the tensor
    \[
      u\oplus v=\min\{1,u+v\},
    \]
    for all $u,v\in [0,1]$. This quantale will be denotes as
    $\overleftarrow{[0,1]}_\oplus$.
  \item\label{d:item:3} The quantales introduced in (\ref{item:2}) and
    (\ref{item:3}) can be more uniformly described using the unit interval
    $[0,1]$ equipped with the usual order $\leqslant$. In fact, $[0,1]$ admits
    several interesting quantale structures, the most important ones to us are
    the minumum $\wedge$, the usual multiplication $*$, and the Lukasiewicz sum
    defined by $u\odot v=\max\{0,u+v-1\}$, for all $u,v\in [0,1]$. The
    corresponding operation $\hom$ is given, respectively, by
    \begin{align*}
      \hom(u,v)&=
      \begin{cases}
        1, & \text{if }u \leq v\\v, & \text{else}
      \end{cases},&
                     \hom(u,v)&=
                     \begin{cases}
                       \min\{\frac{v}{u},1\}, & \text{if }u \neq 0  \\
                       1, & u=0
                     \end{cases},
    \end{align*}
    \[
      \hom(u,v)=\min\{1,1-u+v\}=1-\max\{0,u-v\},
    \]
    for $u,v \in [0,1]$. We will denote these quantales by $\zerone_{\wedge}$,
    $\zerone_{*}$, and $\zerone_{\odot}$, respectively. Then, through the map
    \[
      [0,\infty] \longrightarrow [0,1],\,u\longmapsto e^{-u}
    \]
    with $e^{-\infty}=0$, the quantale $\Pp$ is isomorphic to $\zerone_{*}$ and
    $\Pm$ is isomorphic to $[0,1]_\wedge$. Finally, the quantale
    $\overleftarrow{[0,1]}_\oplus$ of (\ref{item:3}) is isomorphic to the
    quantale $[0,1]_\odot$, via the lattice isomorphism $u \mapsto 1-u$.
  \item Another way to equip the unit interval $\zerone$ with a quantale
    structure is considering the usual order and $\otimes$ is given by the
    \df{nilpotent minimum}
    \[
      u\otimes v=
      \begin{cases}
        \min\{u,v\} & \text{if }u+v> 1,\\
        0 & \text{else}
      \end{cases}
    \]
    for $u,v \in [0,1]$, for which $\hom(u,v)=\max\{1-u,v\}$. This is a
    classical example of a tensor in $\zerone$ that is left continuous but not
    continuous.
  \item The set
    \[
      \ddf = \{f:\zerinf \longrightarrow \zerone \mid
      f(\alpha)=\bigvee_{\beta<\alpha}f(\beta)\text{ for all
      }\alpha\in[0,\infty]\}
    \]
    of left continuous distribution functions, ordered pointwise, is a complete
    lattice. Here the supremum of a family $(h_i)_{i\in I}$ of elements of
    $\ddf$ can be calculated pointwise as
    $h(\alpha)=\bigvee_{i\in I}h_i(\alpha)$, for all $\alpha\in [0,\infty]$. The
    infimum of an arbitrary collection of elements of $\ddf$ cannot be obtained
    by an analogous process since the point-wise infimum of a family of left
    continuous maps need not be left continuous. However, the infimum of a
    family $(f_i)_{i\in I}$ in $\ddf$ is given by
    \[
      \bigwedge_{i\in I}f_i(\alpha)=\sup_{\beta<\alpha}\inf_{i\in I}f_i(\beta),
    \]
    for every $\alpha \in \zerinf$, due to the adjunction $i\dashv c$, where $i$
    is the embedding $ \ddf \to \ORD (\zerinf ,\zerone)$ and
    $c: \ORD (\zerinf ,\zerone) \to \ddf$, such that
    $c(f)(\alpha)=\sup_{\beta <\alpha}f(\beta)$.

    For each of the tensor products $\otimes$ on $[0,1]$ defined in
    (\ref{d:item:3}), $\ddf$ becomes a quantale with
    \[
      f\otimes g (\gamma)=\bigvee_{\alpha+\beta\leqslant \gamma}f(\alpha)\otimes
      g(\beta),
    \]
    for all $\gamma\in [0,\infty]$; the identity is given by
    \[
      f_{0,1}(\alpha)=
      \begin{cases}
        0 & \text{if }\alpha=0,\\
        1 & \text{else}.
      \end{cases}
    \]
    For more information about this quantale we refer to
    \cite{Fla97,HR13,CH16_tmp}.
  \end{enumerate}
\end{examples}

\subsection{Completely distributive lattices}

In this subsection we recall some properties of complete lattices and quantales
which will be useful in the sequel. First of all, we call a quantale
$\V=\quantale$ \df{non-trivial} whenever $k>\bot$. More generally:

\begin{definition}
  The neutral element $k$ of a quantale $\V=\quantale$ is called
  \df{$\vee$-irreducible} whenever $k>\bot$ and, for all $u,v \in \V$,
  $k\leq u\vee v$ implies $k\leq u$ or $k\leq v$.
\end{definition}

For an ordered set $X$, we denote by $\Pdn X$ the complete lattice of down sets
of $X$ ordered by inclusion. The ordered set $X$ can be embedded into $\Pdn X$
by
\[
  \downarrow_X:X \longrightarrow \Pdn X,\,x \longmapsto \downc x=\{y\in X \mid
  y\le x\};
\]
and $X$ is complete if and only if $\downarrow_X:X\to \Pdn X$ has a left adjoint
$\bigvee_X:\Pdn X \to X$. In this paper we will often require that the complete
lattice $\V$ is completely distributive (see \cite{Ran52,Woo04}), therefore we
recall now:

\begin{definition}
  A complete ordered set $X$ is called \df{completely distributive} whenever the
  map $\bigvee_X:\Pdn X \to X$ preserves all infima.
\end{definition}

Hence, since $\Pdn X$ is complete, the lattice $X$ is completely distributive if
and only if $\bigvee_X$ has a left adjoint $ \Downarrow_X:X\to \Pdn X$. We
recall that
\begin{align*}
  \Downarrow_X \dashv \bigvee_X
  && \iff
  && \forall x\in X\, \forall A\in \Pdn X\,.\,(\Downarrow_X x\subseteq A \iff x\leq \bigvee_X A).
\end{align*}

\begin{definition}
  Let $X$ be a complete ordered set $X$. For all $x,y\in X$, \df{$x$ is totally
    below $y$ ($x\ll y$)} whenever, for all $A\in \Pdn X$,
  \begin{align*}
    y\leq \bigvee A \implies x\in A. 
  \end{align*}
\end{definition}

\begin{proposition}
  Let $\ll$ be the totally below relation in a complete ordered set $X$ with
  order relation $\le$. Then, for all $x,y,z \in X$:
  \begin{enumerate}
  \item $x\ll y \implies x\leq y$;
  \item $x\leq y \ll z \implies x\ll z$;
  \item $x\ll y\leq z \implies x \ll z$;
  \item $x\ll y \implies \exists z\in X\,.\,x\ll z \ll y$.
  \end{enumerate}
\end{proposition}

If $X$ is a completely distributive lattice, then, for every $y\in X$,
\[
  \Downarrow y=\bigcap\{A\in \Pdn X \mid y\leq \bigvee A \};
\]
therefore $x\in \Downarrow y$ if and only if $x\ll y$.

\begin{theorem}
  A complete lattice $X$ is completely distributive if and only if every
  $y\in X$ can be expressed as $y=\bigvee \{x\in X \mid x\ll y\}$.
\end{theorem}

\begin{remark}
  A complete ordered set $X$ is completely distributive if and only if $X^\op$
  is so (see \cite{Woo04}).
\end{remark}

\begin{examples}\label{exs:2}
  \begin{enumerate}
  \item The complete lattice $\two$ is completely distributive where $x\ll y$ if
    and only if $y=1$.
  \item The lattices $\zerone$ and $\zerinf$, ordered by $\leqslant$, are
    completely distributive with $\ll$ being the usual smaller relation
    $<$. Similarly, $\overleftarrow{[0,1]}$ and $\overleftarrow{[0,\infty]}$,
    with the ``greater or equal relation'' $\geqslant$, are completely
    distributive where the totally below relation is the larger relation $>$.
  \item In order to show that the complete lattice $\ddf$ is complete
    distributive, it is useful to introduce some special elements that will
    allow a more simplified description of $\ddf$ and of its properties. The
    step functions $\Fne$, with $n\in \zerinf$ and $\varepsilon \in \zerone$,
    are elements of $\ddf$, defined by
    \[
      \Fne(\alpha)=
      \begin{cases}
        0 & \text{  if  } \alpha\leqslant n, \\
        \varepsilon & \text{ if } \alpha>n;
      \end{cases}
    \]
    for all $\alpha\in \zerinf$. It is shown in \cite{Fla97} that, for all
    $f, \Fne \in \ddf$, $\Fne \ll f$ if and only if $\varepsilon < f(n)$. This
    observation allows to write every element $f\in \ddf$ as the supremum of
    those step functions totally bellow $f$:
    $f=\bigvee\{\Fne \in \ddf \mid \Fne \ll f\}$. A complete description of the
    totally below relation on $\ddf$ can be found in \cite{CH16_tmp}.
  \end{enumerate}
\end{examples}

\begin{definition}
  For a quantale $\V=\quantale$, we say that $k$ is \df{approximated} whenever
  the set
  \[
    \Dnw k=\{u\in\V\mid u\ll k\}
  \]
  is directed and $k=\bigvee\Dnw k$.
\end{definition}

We note that in each of the quantales of Examples~\ref{exs:2} the neutral
element is approximated.

\begin{proposition}\label{d:prop:4}
  Let $\V=\quantale$ be quantale where $k$ is approximated. Then $k$ is
  $\vee$-irreducible and
  \[
    k=\bigvee_{u\ll k}u\otimes u.
  \]
\end{proposition}
\begin{proof}
  Assume that $k$ is approximated. First note that $k>\bot$ since, being
  directed, $\Dnw k$ is non-empty. Furthermore, $k$ is $\vee$-irreducible by
  \cite[Remark 4.21]{HR13}, and the second assertion follows from \cite[Theorem
  1.12]{Fla92}.
\end{proof}

\subsection{$\V$-relations}
\label{ssec:V-rel}

For a quantale $\V=\quantale$, a \df{$\V$-relation} $r: X\relto Y$ is a map
$X\times Y \to \V$.  Given $\V$-relations $r: X\relto Y$ and $s:Y\relto Z$,
their composite $s\cdot r :X\relto Z$ is defined by
\[
  s\cdot r(x,z)=\bigvee_{y\in Y}r(x,y)\otimes s(y,z),
\]
and the identity on $X$ is the $\V$-relation $1_X:X\relto X$ given by
\[
  1_X(x,y)=
  \begin{cases}
    k&\text{if }x=y,\\
    \bot, & \text{else}.
  \end{cases}
\]
The resulting category of sets and $\V$-relations is denoted by
$\Rels{\V}$. Similarly to the case of the identity relation, every map
$f:X\to Y$ can be seen as a $\V$-relation $f:X\relto Y$ with
\[
  f(x,y)=
  \begin{cases}
    k & \text{if }f(x)=y,\\
    \bot & \text{else;}
  \end{cases}
\]
this construction defines a functor $\SET\to\Rels{\V}$. We note that this
functor is faithful if and only if $\V$ is non-trivial.

The set $\Rels{\V}(X,Y)$ of $\V$-relations from $X$ to $Y$ is actually a
complete ordered set where the supremum of a family
$(\varphi_i: X \relto Y)_{i\in I}$ is calculated pointwise. Since the tensor
product of $\V$ preserves suprema, for every $\V$-relation $r:X \relto Y$, the
maps $(-)\cdot r:\Rels{\V}(Y,Z)\to \Rels{\V}(X,Z)$ and
$r\cdot (-):\Rels{\V}(Z,X)\to \Rels{\V}(Z,Y)$ preserve suprema as well; which
tells as that $\Rels{\V}$ is actually a quantaloid (see \cite{Ros96a}). In
particular, both maps have right adjoints in $\ORD$.

Here a right adjoint $-\blackleft r$ of $-\cdot r$ must give, for each
$ t:X\relto Z$, the largest $\V$-relation of type $Y\relto Z$ whose composite
with $r$ is less or equal $t$,
\[
  \xymatrix{X\ar[r]|-{\object@{|}}^t \ar[d]|-{\object@{|}}_r & Z\\
    Y\ar@{..>}[ur]|-{\object@{|}}^\leq}
\]
and we call $t \blackleft r$ the \df{extension of $t$ along $r$}. Explicitly,
\[
  t\blackleft r(y,z)=\bigwedge_{x\in X}\hom(r(x,y),t(x,z)).
\]
Similarly, a right adjoint $r\blackright-$ of $r\cdot-$ must give, for each
$t:Z\relto Y$, the largest $\V$-relation of type $Z\relto X$ whose composite
with $r$ is less or equal $t$.
\[
  \xymatrix{Y &
    Z\ar[l]|-{\object@{|}}_t\ar@{..>}[dl]|-{\object@{|}}_\leq\\
    X\ar[u]|-{\object@{|}}^r}
\]
The $\V$-relation $r\blackright t$ is called the \df{lifting of $t$ along $r$},
and can be calculated as
\[
  r\blackright t(z,x)=\bigwedge_{y\in Y}\hom(r(x,y),t(z,y)).
\]

Another important feature which comes from the fact that $\Rels{\V}$ is locally
ordered, is the possibility to define adjoint $\V$-relations: $r:X\relto Y$ is
left adjoint to $s:Y\relto X$ if and only if $1_X\leq s\cdot r$ and
$r\cdot s \leq 1_Y$, which in pointwise notation, gives, for all $x\in X$,
\begin{align*}
  k\leq \bigvee_{y\in Y}r(x,y)\otimes s(y,z)
  && \text{and}
  && \forall y,y'\in Y\,.\,(y\neq y' \implies s(y,x)\otimes r(x,y')=\bot ).
\end{align*}

For each $\V$-relation $r: X\relto Y$ one can consider its opposite
$r^{\circ}:Y\relto X$ given by $r^{\circ}(x,y)=r(y,x)$, for all $x\in X$ and all
$y\in Y$. This operation satisfies
\begin{align*}
  1_X^{\circ} =1_X, && (s\cdot r)^{\circ}=r^{\circ} \cdot s^{\circ},  &&  (r^{\circ})^{\circ}=r,
\end{align*}
and
\[
  r_1\le r_2 \iff r_1^{\circ}\le r_2^{\circ},
\]
for all $r,r_1,r_2:X\relto Y$ and $s:Y\relto Z$. Hence, this construction
defines a locally monotone functor $(-)^\op : \Rels{\V}^\op \to\Rels{\V}$. We
also note that $f\dashv f^\circ$ in $\Rels{\V}$, for every function $f:X\to Y$.

\subsection{$\V$-categories}
\label{ssec:quant-enrich-categ}

We introduce now categories enriched in a quantale $\V$.

\begin{definition}
  Let $\V=\quantale$ be a quantale. A \df{$\V$-category} is a pair $(X,a)$
  consisting of a set $X$ and a $\V$-relation $a:X\relto X$ satisfying
  $1_X\leq a$ and $a \cdot a \leq a$; in pointwise notation:
  \begin{align*}
    k\leq a(x,x) && \text{and}  && a(x,y)\otimes a(y,z)\leq a(x,z),
  \end{align*}
  for all $x,y,z \in X$. A \df{$\V$-functor} $f:(X,a)\to (Y,b)$ between
  $\V$-categories is a map $f:X\to Y$ such that $f\cdot a \leq b\cdot f$;
  equivalently, for all $x, x' \in X$,
  \[
    a(x,x')\leq b(f(x),f(x')).
  \]
\end{definition}

With the usual composition of maps and the identity maps, $\V$-categories and
$\V$-functors provide the category $\Cats{\V}$.  Note that $1_X\le a$ implies
$a\leq a\cdot a$, which implies $a\cdot a=a$, for every $\V$-category
$(X,a)$. The quantale $\V$ is itself a $\V$-category with structure given by
$\hom:\V\times \V \to\V$. To every $\V$-category $(X,a)$ one can associate its
dual $\V$-category $X^\op =(X, a^\circ)$, and this construction defines a
functor
\[
  (-)^\op:\Cats{\V}\longrightarrow\Cats{\V}
\]
commuting with the canonical forgetful functor $\FgtSet{\V}:\Cats{\V}\to\SET$.

\begin{definition}
  A $\V$-category $X$ is called \df{symmetric} whenever $X=X^\op$.
\end{definition}

Due to the fact that the forgetful functor $\FgtSet{\V}:\Cats{\V}\to \SET$ is
topological (see \cite{AHS90,CH03}), the category $\Cats{\V}$ admits all limits
and colimits. Moreover, $\FgtSet{\V}:\Cats{\V}\to \SET$ has a left adjoint and
the free $\V$-category over the one-element set $1=\{\star\}$ is given by
$G=(1,k)$, where $k(\star,\star)=k$. For every set $X$, we have the $X$-fold
power $\V^X$ of the $\V$-category $\V$ whose elements are maps $\varphi:X\to\V$
and, for maps $\varphi_1,\varphi_2:X\to\V$,
\[
  [\varphi_1,\varphi_2]:=\varphi_2\blackleft\varphi_1=\bigwedge_{x\in
    X}\hom(\varphi_1(x),\varphi_2(x))
\]
describes the $\V$-categorical structure of $\V^X$. Another example is the
product of two $\V$-categories $(X,a)$ and $(Y,b)$, which is the $\V$-category
$X\times Y=(X\times Y,d)$, where, for $(x,y),(x',y')\in X\times Y$,
$d((x,y),(x',y'))=a(x,x')\wedge b(y,y')$. However one can also consider the
structure $a\otimes b$ on $X\times Y$:
$a\otimes b ((x,y),(x',y'))=a(x,x')\otimes b(y,y')$. Both products are
commutative and associative but the neutral objects differ in general:
$(1,\top)$ is the neutral object for the first product while $G=(1,k)$ is the
neutral object for the second.

We consider now the quantales of Examples~\ref{exs:1}.
\begin{examples}\label{exs:V-Cats}
  \begin{enumerate}
  \item The objects of $\Cats{\two}$ are ordered sets (that is, sets equipped
    with a reflexive and transitive binary relation) and the morphisms are
    monotone maps; thus, $\Cats{\two}\simeq \ORD$.
  \item A $\Pp$-category is a generalised metric space in the sense of
    \cite{Law73} and a $\Pp$-functor is a non-expansive map. We write $\MET$ for
    the resulting category, that is, $\Cats{\Pp}\simeq\MET$. Due to the lattice
    isomorphism $\Pp\simeq \zerone_{*}$, also $\Cats{\zerone_{*}}\simeq \MET$.
    Similarly, for $\V =\Pm$, a $\V$-category is a (generalised) ultrametric
    space and, since $\Pm \simeq \zerone_\wedge$, we have
    $\Cats{\Pm}\simeq\Cats{\zerone_\wedge}\simeq \UMET$. Finally, we can
    interpret $\overleftarrow{[0,1]}_\oplus$-categories and
    $[0,1]_\odot$-categories as bounded-by-$1$ metric spaces and
    $[0,1]_\odot$-functors as non-expansive maps, so that
    $\Cats{\overleftarrow{[0,1]}_\oplus}\simeq\Cats{[0,1]_\odot}\simeq\BMET$.
  \item A $\ddf$-category consists of a set equipped with a structure
    $a:X\times X \to \ddf$ such that, for all $x,y,z\in X$ and $t\in \zerinf$:
    \begin{align*}
      1\leqslant a(x,y)(t)
      && \text{and}
      && \bigvee_{q+r\leqslant t} a(x,y)(q)\otimes a(y,z)(r)\leqslant a(x,z)(t),
    \end{align*}
    and a $\ddf$-functor $f:(X,a)\to (Y,b)$ satisfies
    $a(x,y)(t)\leqslant b(f(x),f(y))(t)$, for $x,y\in X$ and $t\in \zerinf$.
    Therefore the category $\Cats{\ddf}$ is isomorphic to the category of
    (generalised) probabilistic metric spaces $\PROBMET$. The classical
    definition of probabilistic metric space (see \cite{Men42,SS83}) demands
    that $(X,a)$ is separated (see Definition~\ref{d:def:3}), symmetric and
    finitary ($a(x,y)\in \ddf$ should be finite for all $x,y \in X$). A detailed
    study of probabilistic metric spaces as enriched categories can be found in
    \cite{HR13}.
  \end{enumerate}
\end{examples}

\begin{definition}
  Let $\V_1$ and $\V_2$ be quantales, we write $\otimes$ for the multiplication
  in both $\V_1$ and $\V_2$, and $k_1$ denotes the neutral element of $\V_1$ and
  $k_2$ the neutral element of $\V_2$. A \df{lax quantale morphism}
  $\varphi:\V_1\to\V_2$ is a monotone map between the underlying ordered sets
  satisfying
  \begin{align*}
    k_2\le\varphi(k_1)
    &&\text{and}
    && \varphi(u)\otimes\varphi(v)\le\varphi(u\otimes v),
  \end{align*}
  for all $u,v\in\V_2$.
\end{definition}
These properties ensure that the mapping
\[
  (X,a) \longmapsto (X,\varphi a)
\]
sends $\V_1$-categories to $\V_2$-categories; hence, this construction defines a
functor
\[
  B_\varphi:\Cats{\V_1} \longrightarrow\Cats{\V_2}
\]
which commutes with the forgetful functors to $\SET$.

\begin{examples} \label{exs:quantale_morphs}
  \begin{enumerate}
  \item The identity map on $[0,\infty]$ defines a lax quantale morphism
    \[
      \Pm \longrightarrow\Pp,
    \]
    and the map $[0,\infty]\to [0,1],\,u\mapsto\min(u,1)$ gives a lax quantale
    morphism
    \[
      \Pp \longrightarrow \overleftarrow{[0,1]}_\oplus.
    \]
    The corresponding functors produce the canonical chain of functors
    \[
      \UMET \longrightarrow\MET \longrightarrow\BMET.
    \]
  \item \label{functors_in_delta}The quantale $\Pp$ embeds canonically into
    $\ddf$ via $I_\infty:\Pp \to \ddf$, taking an element $n\in \zerinf$ to
    $f_{n,1} \in \ddf$. This map is a lax quantale morphism and it
    admits a right and a left adjoint
    \[
      \xymatrix{\Pp\ar[rr]|-{I_\infty} & &
        \ddf\ar@/^3.5ex/[ll]_\perp^-{P_\infty}\ar@/_3.5ex/[ll]^\perp_-{O_\infty}}
    \]
    with $P_\infty(f)=\inf\{n\in \zerinf \mid f(n)=1\} $ and
    $O_\infty(f)=\sup\{n\in \zerinf\mid f(n)=0\}$, for all $f\in \ddf$ with
    $P_\infty$ and $O_\infty$ being lax quantale morphisms. These lax morphisms
    induce adjoint functors
    \[
      \xymatrix{\MET\ar[rr]|-{I_\infty} &&
        \PROBMET.\ar@/^3.5ex/[ll]_\perp^-{P_\infty}\ar@/_3.5ex/[ll]^\perp_-{O_\infty}}
    \]
    between the categories $\MET$ and $\PROBMET$.
  \end{enumerate}
\end{examples}

For every quantale $\V$, the canonical map
\[
  i:\two \longrightarrow \V,\,0 \longmapsto \bot,\, 1 \longmapsto k
\]
is a lax quantale morphism, which induces the functor
\[
  B_i:\ORD \longrightarrow\Cats{\V}.
\]
The monotone map $i:\two\to\V$ has a right adjoint
\[
  p:\V \longrightarrow \two,\, v \longmapsto
  \begin{cases}
    1 & \text{if }v\ge k,\\
    0 & \text{else}
  \end{cases}
\]
which is a lax morphism of quantales too and induces the functor
$B_p:\Cats{\V}\to\ORD$; explicitly,
\[
  x\leq y \iff k \leq a(x,y),
\]
for all elements $x,y$ of a $\V$-category $X$.

\begin{definition}\label{d:def:3}
  A $\V$-category $X=(X,a)$ is called \df{separated} whenever $B_pX$ is
  separated; that is, for all $x,y\in X$, $k\le a(x,y)$ and $k\le a(y,x)$ imply
  $x=y$.
\end{definition}

\subsection{$\V$-distributors}

Besides $\V$-functors, there is another important type of morphisms between
categories, called $\V$-distributors. The notion of distributor was introduced
by B{\'e}nabou in the 1960's and provides ``a generalisation of relations
between sets to `relations between (small) categories' '' (see \cite{Ben00}).

\begin{definition}
  For $\V$-categories $X=(X,a)$ and $Y=(Y,b)$, a \df{$\V$-distributor}
  $\varphi :X\modto Y$ is a $\V$-relation $\varphi: X\relto Y$ compatible with
  both structures, meaning that $ \varphi \cdot a \leq \varphi$ and that
  $b\cdot \varphi \leq \varphi$.
\end{definition}

In fact, these inequalities are equalities due to the reflexivity of $a$ and
$b$. Thus the identity distributor on $(X,a)$ is actually $a$ and, considering
the composition of $\V$-relations, we obtain the category $\Dists{\V}$. We also
note that a $\V$-relation $\varphi: X \relto Y$ is a $\V$-distributor precisely
when $\varphi: X^\op\otimes Y \to \V$ is a $\V$-functor (see \cite{Law73}).

For $\V$-categories $X$ and $Y$, the subset
\[
  \Dists{\V}(X,Y)\hookrightarrow\Rels{\V}(X,Y)
\]
is closed under suprema; hence, the supremum of a family
$(\varphi_i:X\modto Y)_{i\in I}$ can be calculated pointwise. As in
Subsection~\ref{ssec:V-rel}, for a $\V$-distributor $\varphi: X\modto Y$, both
maps $(-)\cdot \varphi$ and $ \varphi \cdot (-)$ have right adjoint given,
respectively, by the extension and lifting along $\varphi$.

Every $\V$-functor $f:(X,a)\to (Y,b)$ induces a pair of $\V$-distributors
$f_*:(X,a)\modto (Y,b)$ and $f^*:(Y,b)\modto (X,a)$ given by $f_*=b\cdot f$ and
$f^*=f^\circ\cdot b$; in pointwise notation, for $x\in X$ and $y\in Y$,
\begin{align*}
  f_*(x,y)=b(f(x),y) && \text{and} && f^*(y,x)= b(y,f(x)),
\end{align*}
which characterise the functors $(-)_*:\Cats{\V}\to\Dists{\V}$ and
$(-)^*\Cats{\V}\to \Dists{\V}^{\op}$. An important fact about these induced
$\V$-distributors is that they form an adjunction $f_*\dashv f^*$ in
$\Dists{\V}$ since
\[
  f^*\cdot f_*=f^\circ\cdot b\cdot b\cdot f= f^\circ\cdot b\cdot f \ge
  f^\circ\cdot f\cdot a\geq a
\]
and
\[
  f_*\cdot f^*=b\cdot f \cdot f^\circ \cdot b\leq b\cdot b=b.
\]
For the particular case of a $\V$-functor of the form $x:G\to X$ we obtain
$x_*=a(x,-)$ and $x^*=(-,x)$.

\begin{definition}
  A $\V$-functor $f:(X,a)\to (Y,b)$ is called \df{fully faithful} whenever
  $f^*\cdot f_*=a$, and $f$ is called \df{fully dense} whenever
  $f_*\cdot f^*=b$.
\end{definition}

The underlying order of $\V$-categories extends point-wise to an order relation
between $\V$-functors. This order relation can be equivalently described using
$\V$-distributors: for $\V$-functors $f,g:(X,a)\to(Y,b)$,
\[
  f\leq g\iff f^*\leq g^*\iff g_*\leq f_*.
\]
Furthermore, the composition from either sides preserves this order, and
therefore $\Cats{\V}$ is actually an ordered category. An important consequence
is the possibility to define adjoint $\V$-functors: a pair of $\V$-functors
$f:(X,a) \to (Y,b)$ and $g: (Y,b)\to (X,a)$ forms an adjunction, $f\dashv g$,
whenever, $1_X\leq g\cdot f$ and $\cdot fg\leq 1_Y$. Since
\[
  f\dashv g\iff g_*\dashv f_*\iff f_*=g^*,
\]
$f\dashv g$ if and only if, for all $x\in X$ and $y\in Y$,
$a(x,g(y))=b(f(x),y)$.

\subsection{Cauchy complete $\V$-categories}

In 1973, Lawvere \cite{Law73} proved that a metric space $X$ is Cauchy complete
if and only if every adjunction $\varphi \dashv \psi : Y\modto X$ of
$\Pp$-distributors is of the form $f_*\dashv f^*$, for some non-expansive map
$f:X\to Y$. This observation motivates the following nomenclature.

\begin{definition}
  Let $\V=\quantale$ be a quantale. A $\V$-category $(X,a)$ is \df{Cauchy
    complete} if every adjunction of $\V$-distributors
  $(\varphi:X\modto Y)\dashv(\psi:Y\modto X)$ is representable, meaning that
  there is a $\V$-functor $f:Y\to X$ such that $\varphi=f_*$ and $\psi=f^*$.
\end{definition}

Although the definition requires the representability of every adjunction, it is
enough to consider the case $Y=G$. Thus, a $\V$-category $X$ is Cauchy complete
if and only if every adjunction $(\varphi:G\modto X)\dashv(\psi:X\modto G)$ is
representable by some $x\in X$.

Subsequent developments established conditions under which results relating
Cauchy sequences, convergence of sequences, adjunctions of distributors and
representability can be generalised to $\Cats{\V}$ (see
\cite{Fla92,HT10,HR13,CH09,Cha09}). In this subsection we will present some of
these notions and results in order to recall that, under some light conditions
on the quantale $\V$, Lawvere's notion of complete $\V$-categories can be
equivalently expressed with Cauchy sequences.

In order to talk about convergence, we need to introduce first some topological
notions. Here, for a quantale $\V=\quantale$, a $\V$-category $(X,a)$ and a
subset $M\subseteq X$, the \df{L-closure} $\overline{M}$ of $M$ is given by the
collection of all $x\in X$ which \emph{represent adjoint distributors on
  $M$}. More precisely, $x\in\overline{M}$ whenever
$i^*\cdot x_*\dashv x^*\cdot i_*$, where $i:M\hookrightarrow X$ is the inclusion
$\thU$-functor. In more elementary terms, we have:

\begin{proposition}\label{prop:4}
  Let $(X,a)$ be a $\V$-category, $M\subseteq X$ and $x\in X$. Then the
  following assertions are equivalent.
  \begin{tfae}
  \item $x\in\overline{M}$.
  \item $k\le\bigvee_{z\in M}a(x,z)\otimes a(z,x)$.
  \end{tfae}
\end{proposition}

The proposition above also shows that $(X,a)$ and $(X,a)^\op$ induce the same
closure operator on the set $X$.

\begin{proposition}
  Let $\V=\quantale$ be a quantale. For a $\V$-functor $f:X\to Y$ and
  $M,M'\subseteq X$, $N\subseteq Y$, one has:
  \begin{enumerate}
  \item $M\subseteq\overline{M}$ and $M\subseteq M'$ implies
    $\overline{M}\subseteq\overline{M'}$.
  \item $\overline{\overline{M}}=\overline{M}$.
  \item $f(\overline{M})\subseteq \overline{f(M)}$ and
    $f^{-1}(\overline{N})\supseteq\overline{f^{-1}(N)}$.
  \item If $k$ is $\vee$-irreducible, then
    $\overline{M\cup M'}=\overline{M}\cup\overline{M'}$ and
    $\overline{\varnothing}=\varnothing$.
  \end{enumerate}
\end{proposition}
\begin{corollary}\label{d:cor:4}
  If $k$ is $\vee$-irreducible in $\V$, then the L-closure operator defines a
  topology on $X$ such that every $\V$-functor becomes continuous. Hence, in
  this case the L-closure defines a functor $\Cats{\V}\to\TOP$.
\end{corollary}

Equipped with a closure operator, there is a notion of convergence in a
$\V$-category. In particular, a sequence $s=(x_n)_{n\in\N}$ in a $\V$-category
$(X,a)$ converges to $x\in X$ if and only if
$k\leq \bigvee_{m\in M}a(x,x_m)\otimes a(x_m,x)$ for all infinite subset $M$ of
$\N$.

We recall the following definition from \cite{Wag94}.

\begin{definition}
  Let $\V=\quantale$ be a quantale. A sequence $s=(x_n)_{n\in\N}$ in a
  $\V$-category $(X,a)$ is \df{Cauchy} if $k\leq \Cauchy (s)$, were
  $\Cauchy (s)= \bigvee_{N\in\N}\bigwedge_{n,m\geq N}a(x_n,x_m)$.
\end{definition}

Every sequence $s=(x_n)_{n\in \N}$ in a $\V$-category $(X,a)$ induces
$\V$-distributors $\varphi_s:G\modto X$ and $\psi_s:X\modto G$ defined as
\begin{align*}
  \varphi_s(x)&=\bigvee_{N\in\N}\bigwedge_{n\geq N}a(x_n,x)
  &\text{and}&&
                \psi_s(x)&=\bigvee_{N\in\N}\bigwedge_{n\geq N}a(x,x_n),
\end{align*}
for all $x\in X$. The relation between these concepts is stated in the following
result (see \cite{HR13}).

\begin{theorem}
  A sequence $s$ in a $\V$-category $X$ is Cauchy if and only if
  $\varphi_s\dashv\psi_s$ in $\Dists{\V}$.
\end{theorem}

\begin{theorem}
  \label{theor:cond_adj_is_induced_by_Cauchyseq}
  Let $\V$ be a quantale where $k$ is terminal and assume that there is a sequence
  $(u_n)_{n\in \N}$ in $\V$ satisfying:
  \begin{enumerate}
  \item $\bigvee_{n\in\N}u_n=k$,
  \item $\forall n\in\N, u_n\ll k$,
  \item $\forall n\in\N, u_n\leq u_{n+1}$.
  \end{enumerate}
  Then every adjunction of $\V$-distributors of the form
  $\varphi \dashv \psi :X\modto G$ is induced by a Cauchy sequence $s$, that is,
  $\varphi_s=\varphi$ and $\psi_s=\psi$.
\end{theorem}

\begin{theorem}
  Let $\V$ be a quantale where $k$ is terminal and $s$ a Cauchy sequence in a
  $\V$-category $X$. Then $s$ converges to $x$ if and only if $\varphi_s=x_*$
  and $\psi_s=x^*$.
\end{theorem}

\begin{theorem}
  Under the conditions of Theorem \ref{theor:cond_adj_is_induced_by_Cauchyseq},
  a $\V$-category is Cauchy complete if and only if every Cauchy sequence
  converges.
\end{theorem}

\section{Combining convergence and $\V$-categories}
\label{sec:comb-conv-v}

In this section we study $\V$-categories equipped with a compatible convergence
structure. As we explained in Section~\ref{sec:introduction}, this study has its
roots in Nachbin's ``Topology and Order'' \cite{Nac50} as presented in
\cite{Tho09}. We recall the notion of topological theory $\thU$ \cite{Hof07},
which provides enough structure to extent the ultrafilter monad $\mU$ to a monad
on $\Cats{\V}$; the algebras for this monad we designate as
\emph{$\V$-categorical compact Hausdorff spaces}. We also recall the notions of
$\thU$-category and $\thU$-functors and the comparison between $\V$-categorical
compact Hausdorff spaces and $\thU$-categories, which can be already found in
\cite{Tho09}. In the last subsection we use the closure operator on $\Cats{\V}$
introduced in \cite{HT10} to define \emph{compact} $\V$-categories, and show,
under some conditions on $\V$, that compact separated $\V$-categories provide
examples of $\V$-categorical compact Hausdorff spaces.

\subsection{The ultrafilter monad}
\label{ssec:uf-monad}

Given a category $\catA$, a \df{monad} on $\catA$ is a triple $\mT=\monad$
consisting of a functor $T:\catA \to \catA$ and natural transformations
$e:1\to T$ and $m:T^2 \to T$ such that the diagrams
\[
  \xymatrix{TTT \ar[r]^{m_{T}} \ar[d]_{T_m} & TT \ar[d]_m & T \ar[l]_{e_T} \ar[d]^{T_e} \ar[ld]_{1_T}\\
    TT \ar[r]_m & T & TT \ar[l]^m }
\]
commute. For a monad $\mT=\monad$, a \df{$\mT$-algebra} $(X,\alpha)$ is an
object $X$ of $\catA$ together with a map $\alpha: TX\to X$ satisfying
\[
  \xymatrix{TT X \ar[r]^{m_X} \ar[d]_{T\alpha} & TX \ar[d]_{\alpha} & X \ar[l]_{e_X} \ar[ld]^{1_X}\\
    TX \ar[r]_{\alpha} & X & }
\]
A morphism between $\mT$-algebras $f:(X,\alpha)\to (Y,\beta)$ is a map
$f:X\to Y$ such that the diagram
\[
  \xymatrix{TX \ar[r]^{Tf} \ar[d]_{\alpha} & TY \ar[d]^{\beta}\\
    X \ar[r]_{f} & Y }
\]
commutes.  $\mT$-algebras and their morphisms compose the Eilenberg-Moore
category $\catA^{\mT}$ of $\mT$. The forgetful functor
$G^\mT:\catA^{\mT}\to \catA$ has a right adjoint $F^\mT:\catA \to \catA^{\mT}$
that takes an object $X$ of $\catA$ to the $\mT$-algebra $(TX,m_X)$. For more
information on monads we refer to \cite{MS04}.

Every adjunction $(F\dashv G, \eta, \varepsilon):\catA \rightleftarrows\catX$
originates the monad $\monad$ on $\catA$ given by $T=G\cdot F$, $e=\eta$ and
$m=G{\varepsilon_F}$. Of particular interest to us is the ultrafilter monad
$\mU =\umonad$ which is induced by the adjunction
\[
  \BOOLE^\op \adjunct{\SET (-,\two)}{\BOOLE(-,\two)}\SET.
\]
Here the functor $U:\SET \to \SET$ takes a set $X$ to the set $UX$ of
ultrafilters on $X$ and, for a map $f:X\to Y$ and $\fx\in UX$,
$Uf(\fx)=\{A\subseteq Y \mid f^{-1}(A)\in\fx\}$. The unit $e_X:X\to UX$ on $X$
sends $x\in X$ to the principal ultrafilter $\doo{x}$ on $X$, and the
multiplication $m_X:U^2X \to UX$ is characterised, for every $\fX \in U^2X$, by
$m_X(\fX)=\{A\in UX\mid A^\#\in \fX\}$, where $A^\#=\{\fx\in UX\mid A\in
\fx\}$. The Eilenberg-Moore category of $\mU$, $\SET^{\mU}$, is equivalent to
the category of compact Hausdorff topological spaces and continuous maps (see
\cite{Man69}).

The following result (see \cite{Sto38}) ensures the existence of certain
ultrafilters and will be very important for later usage.
\begin{lemma}
  Let $\ff$ be a filter and $\fj$ be an ideal on a set $X$ such that
  $\ff\cap\fj= \varnothing$. Then there is an ultrafilter $\fr$ that extends
  $\ff$ and excludes $\fj$; that is, $\ff\subseteq \fr$ and
  $\fj\cap\fr =\varnothing$.
\end{lemma}

\subsection{Ultrafilter theories}
\label{ssec:Top-theories}

In this paper we consider a particular case $\thU=\utheory$ of a topological
theory (in the sense of \cite{Hof07}) based on the ultrafilter monad
$\mU=\umonad$, a quantale $\V=\quantale$ and a map $\xi:U\V\to\V$. Here we
require $\utheory$ to satisfy all the axioms of the definition of a
\emph{strict} topological theory with the exception of the axiom regarding the
tensor product $\otimes$ of $\V$, for which it is enough to have lax
continuity. We call such a theory an \df{ultrafilter theory}. More in detail:
\begin{itemize}
\item the map $\xi:U\V\to\V$ is the structure of an Eilenberg--Moore algebra on
  $\V$,
  \begin{align*}
    \xymatrix{X\ar[r]^{e_X}\ar[dr]_{1_X} & UX\ar[d]^\xi\\ & X} && \xymatrix{UUX\ar[d]_{U\xi}\ar[r]^{m_X} & UX\ar[d]^\xi\\ UX\ar[r]_\xi & X}
  \end{align*}
  that is, $\xi:U\V\to\V$ is the convergence of a compact Hausdorff topology on
  $\V$;
\item The tensor product is ``laxly continuous'':
  \[
    \xymatrix{U(\V\times\V)\ar[d]_{\langle\xi U\pi_1,\xi
        U\pi_2\rangle}\ar[r]^{U\otimes}\ar@{}[dr]|\le & U\V\ar[d]^\xi\\
      \V\times\V\ar[r]_\otimes & \V}
  \]
\item $\xi$ is ``compatible with suprema in $\V$'' as specified in condition
  (Q$_{\bigvee}$) in \cite{Hof07}.
\end{itemize}

We call a theory $\thU$ satisfying even equality in the axiom involving the
tensor product a \df{strict ultrafilter theory}. We note that every ultrafilter
theory where $\otimes=\wedge$ is strict. Furthermore, $\thU$ is called
\df{compatible with finite suprema} whenever the diagram
\[
  \xymatrix{U(\V\times\V)\ar[d]_{\langle\xi U\pi_1,\xi
      U\pi_2\rangle}\ar[r]^-{U\vee} & U\V\ar[d]^\xi\\
    \V\times\V\ar[r]_\vee & \V}
\]
commutes. Note that we do not need to impose a condition on the empty supremum
since, for every ultrafilter theory, the diagram
\[
  \xymatrix{U1\ar[r]^{U\bot}\ar[d]_{} & U\V\ar[d]^{\xi} \\ 1\ar[r]_{\bot} & \V}
\]
commutes. For $u\in\V$, we consider the map
\[
  t_u:\V \longrightarrow\V,\,v \longmapsto u\otimes v.
\]
An ultrafilter theory $\thU=\utheory$ is called \df{pointwise strict} whenever,
for all $u\in\V$, the diagram
\[
  \xymatrix{U\V\ar[r]^{Ut_u}\ar[d]_{\xi} & U\V\ar[d]^{\xi} \\
    \V\ar[r]_{t_u} & \V}
\]
commutes. Clearly, every strict ultrafilter theory is pointwise strict.

The following result (see \cite[Theorem~3.3]{Hof07}) provides examples of
ultrafilter theories.

\begin{theorem}\label{thm:6}
  For every completely distributive quantale $\V$, the map
  \[
    \xi: U\V \longrightarrow\V,\,\fv \longmapsto
    \bigwedge_{A\in\fv}\bigvee_{u\in A}u
  \]
  defines an ultrafilter theory $\utheory$.
\end{theorem}

Somehow surprisingly, the formula above depends only on the lattice structure of
$\V$; moreover, it is \emph{self-dual} in the sense that
\[
  \xi(\fv) =\bigwedge_{A\in\fv}\bigvee_{u\in A}u
  =\bigvee_{A\in\fv}\bigwedge_{u\in A}u,
\]
for all $\fv\in U\V$. For the lattices $\V=\two$, $\V=[0,1]$, $\V=[0,\infty]$
and $\V=\ddf$ we denote the corresponding map $\xi:U\V\to\V$ by $\xi_\two$,
$\xi_{[0,1]}$, $\xi_{[0,\infty]}$ and $\xi_\ddf$, respectively.

\begin{proposition}\label{d:prop:1}
  Let $\V$ be a completely distributive quantale and $\xi:U\V\to\V$ as in
  Theorem~\ref{thm:6}. Then $\thU=\utheory$ is compatible with finite suprema.
\end{proposition}
\begin{proof}
  Just apply Theorem~\ref{thm:6} to the quantale $\V^\op$ with tensor product
  given by binary suprema $\vee$ in $\V$. Here we use that also the lattice
  $\V^\op$ is completely distributive and therefore in particular a frame.
\end{proof}

\begin{examples}\label{exs:3}
  According to the quantales introduced in Examples~\ref{exs:1}, and keeping in
  mind Examples~\ref{exs:2}, we have the following examples of ultrafilter
  theories.
  \begin{enumerate}
  \item For $\V=\two$, the convergence of Theorem~\ref{thm:6} corresponds to the
    discrete topology on $\two$. We denote this theory as $\thU_\two$.
  \item For the quantales based on the lattices $[0,1]$ and $[0,\infty]$, the
    convergence of Theorem~\ref{thm:6} corresponds to the usual Euclidean
    topology. We denote the corresponding theories by $\thU_{\Pp}$,
    $\thU_{\Pm}$, $\thU_{[0,1]_*}$, $\thU_{[0,1]_\wedge}$, and
    $\thU_{[0,1]_{\odot}}$, respectively.
  \item We will denote the ultrafilter theory based on the quantale $\ddf$ and
    on the the convergence of Theorem~\ref{thm:6} by $\thU_{\ddf}$.
  \end{enumerate}
\end{examples}

\begin{remark}
  For each of the quantales $\V=\two$, $\V=[0,1]$ and $\V=[0,\infty]$, the
  theory obtained from Theorem~\ref{thm:6} is strict. However, we do not know if
  there is a strict ultrafilter theory involving the quantale $\ddf$ of
  distribution functions.
\end{remark}

\begin{definition}
  Let $\thU_1=\utheoryone$ and $\thU_2=\utheorytwo$ be ultrafilter theories and
  $\varphi:\V_1\to\V_2$ be a lax quantale morphism. Then $\varphi$ is
  \df{compatible} with $\thU_1$ and $\thU_2$ whenever, for all $\fv\in U\V_1$,
  $\xi_2\cdot U\varphi(\fv)\le \varphi\cdot\xi_1(\fv)$.
  \[
    \xymatrix{U\V_1\ar[r]^{U\varphi}\ar[d]_{\xi_1}\ar@{}[dr]|\ge &
      U\V_2\ar[d]^{\xi_2} \\ \V_1\ar[r]_{\varphi} & \V_2}
  \]
\end{definition}

For instance, for every ultrafilter theory $\thU=\utheory$, the canonical
map $i:\two\to\V$ (see Subsection~\ref{ssec:quant-enrich-categ}) is a lax
quantale morphism making the diagram
\[
  \xymatrix{U\two\ar[r]^{Ui}\ar[d]_{\xi_\two} & U\V\ar[d]^{\xi} \\
    \two\ar[r]_{i} & \V}
\]
commutative; hence $i$ is compatible with $\thU_\two$ and $\thU$.
\begin{lemma}\label{d:lem:3}
  Let $\thU=\utheory$ be an ultrafilter theory where $k$ is terminal in
  $\V$. Then the right adjoint $p:\V\to\two$ of $i$ is compatible with $\thU$
  and $\thU_\two$.
\end{lemma}
\begin{proof}
  Let $\fv\in U\V$ and assume that $\xi_\two(Up(\fv))=1$. Then $Up(\fv)=\doo{1}$
  and therefore $\upc k\in\fv$. If $k$ is the top-element of $\V$, then
  $\{k\}\in\fv$ and consequently $\xi(\fv)=k$.
\end{proof}

\begin{lemma}\label{d:lem:4}
  Let $\thU_1=\utheoryone$ and $\thU_2=\utheorytwo$ be ultrafilter theories
  where $\xi_1,\xi_2$ are as in Theorem~\ref{thm:6}.  Assume that
  $\varphi:\V_1\to\V_2$ is a lax quantale morphism preserving codirected
  infima. Then $\varphi$ is compatible with $\thU_1$ and $\thU_2$.
\end{lemma}
\begin{proof}
  Let $\fv\in U\V_1$. Then
  \begin{align*}
    \xi_2(U\varphi(\fv))
    &=\bigwedge_{A\in\fv}\bigvee_{u\in A}\varphi(u)\\
    &\leq \bigwedge_{A\in\fv}\varphi\left(\bigvee_{u\in A}u\right)\\
    &=\varphi\left(\bigwedge_{A\in\fv}\bigvee_{u\in A}u\right)
      =\varphi(\xi_1(\fv)).\qedhere
  \end{align*}
\end{proof}
The result above applies in particular when $\varphi:\V_1\to\V_2$ is right
adjoint. For instance, if $\thU=\utheory$ is an ultrafilter theory where $\V$ is
completely distributive and $\xi:U\V\to\V$ is as in Theorem~\ref{thm:6}, then
$p:\V\to\two$ is compatible with $\thU$ and $\thU_\two$ since $i\dashv p$.

\begin{examples} \label{exs:functors_Ucompatible} Recall the chain
  $O_\infty \dashv I_\infty \dashv P_\infty$ of adjoint lax quantale morphisms
  introduced in Example~\ref{exs:quantale_morphs}
  (\ref{functors_in_delta}). Since $I_\infty$ and $P_\infty$ are both right
  adjoints, they are compatible with the ultrafilter theories $\thU_{\Pp}$ and
  $\thU_{\ddf}$.
\end{examples}

\subsection{Extending the monad}
\label{sec:extending-monad}

Given an ultrafilter theory $\thU = \utheory$, we extend the functor
$U:\SET\to\SET$ to a lax functor $\Uxi$ on $\Rels{\V}$ by putting $\Uxi X=UX$
for each set $X$ and
\begin{align*}
  \Uxi  r:UX\times UY & \longrightarrow \V \\
  (\fx,\fy) &\longmapsto\bigvee\left\{\xi\cdot Ur(\fw)\;\Bigl\lvert\;\fw\in U(X\times Y), U\pi_{X}(\fw)=\fx, U\pi_{Y}(\fw)=\fy\right\}
\end{align*}
for each $\V$-relation $r:X\times Y\to\V$. Then we have:

\begin{theorem}\label{thm:1}
  Let $\thU=\utheory$ be an ultrafilter theory. Then the following assertions
  hold.
  \begin{enumerate}
  \item For each $\V$-relation $r:X\relto Y$, $\Uxi (r^\circ)=\Uxi (r)^\circ$
    (and we write $\Uxi r^\circ$).
  \item For each function $f:X\to Y$, $Uf= \Uxi f$ and
    $(Uf)^\circ= \Uxi (f^\circ)$.
  \item For each $\V$-relation $r:X\relto Y$ and functions $f:A\to X$ and
    $g:Y\to Z$,
    \begin{align*}
      \Uxi (g\cdot r)&=Ug\cdot \Uxi r &\text{and}&& \Uxi (r\cdot f)&= \Uxi r\cdot Uf.
    \end{align*}
  \item For all $\V$-relations $r:X\relto Y$ and $s:Y\relto Z$,
    $\Uxi s\cdot \Uxi r\le \Uxi (s\cdot r)$. We have even equality if $\thU$ is
    a strict theory.
  \item Then $e$ becomes an op-lax natural transformation $e:1\to\Uxi$ and $m$ a
    natural transformation $m:\Uxi\Uxi\to\Uxi$, that is, for every $\V$-relation
    $r:X\relto Y$ we have
    \begin{align*}
      e_{_Y}\cdot r\le \Uxi r\cdot e_{_X}, && m_{Y}\cdot \Uxi \Uxi r= \Uxi r\cdot m_{X}.\\
      \xymatrix{X\ar[r]^-{e_{X}}\ar[d]|{\object@{|}}_r\ar@{}[dr]|{\le} & \Uxi X\ar[d]|{\object@{|}}^{\Uxi r}\\ Y\ar[r]_-{e_{Y}} & {U_{\!\xi}}Y} &&
                                                                                                                                                   \xymatrix{\Uxi \Uxi X\ar[r]^-{m_{X}}\ar[d]|{\object@{|}}_{\Uxi \Uxi r}  & \Uxi X\ar[d]|{\object@{|}}^{\Uxi r}\\ \Uxi \Uxi  Y\ar[r]_-{m_{Y}} & \Uxi Y}
    \end{align*}
  \end{enumerate}
\end{theorem}

\subsection{$\V$-categorical compact Hausdorff spaces}
\label{sec:v-categ-comp}

Based on the lax extension of the $\SET$-monad $\mU=\umonad$ to $\Rels{\V}$
described in Subsection~\ref{sec:extending-monad}, the $\SET$-monad $\mU$ admits
a natural extension to a monad on $\Cats{\V}$, in the sequel also denoted as
$\mU=\umonad$ (see \cite{Tho09}). Here the functor $U:\Cats{\V}\to\Cats{\V}$
sends a $\V$-category $(X,a_0)$ to $(UX,\Uxi a_0)$, and with this definition
$e_X:X\to UX$ and $m_X:UUX\to UX$ become $\V$-functors for each $\V$-category
$X$.

\begin{definition}
  Let $\thU=\utheory$ be an ultrafilter theory. An Eilenberg--Moore algebras for
  the monad $\mU$ on $\Cats{\V}$ is called \df{$\V$-categorical compact
    Hausdorff space}.
\end{definition}

Hence, a $\V$-categorical compact Hausdorff space can be described as a triple
$(X,a_0,\alpha)$ where $(X,a_0)$ is a $\V$-category and $\alpha:UX\to X$ is the
convergence of a compact Hausdorff topology on $X$ such that
$\alpha:U(X,a_0)\to(X,a_0)$ is a $\V$-functor. For $\mU$-algebras
$(X,a_0,\alpha)$ and $(Y,b_0,\beta)$, a map $f:X\to Y$ is a homomorphism
$f:(X,a_0,\alpha)\to(Y,b_0,\beta)$ precisely if $f$ preserves both structures,
that is, whenever $f:(X,a_0)\to(Y,b_0)$ is a $\V$-functor and
$f:(X,\alpha)\to(Y,\beta)$ is continuous. Since the extension $\Uxi$ of $U$
commutes with the involution $(-)^\circ$, with $X=(X,a_0,\alpha)$ also
$(X,a_0^\circ,\alpha)$ is a $\V$-categorical compact Hausdorff space. It follows
from \cite[Lemma 3.2]{Hof07} that the $\V$-category $(\V,\hom)$ combined with
the $\mU$-algebra structure $\xi$ induces the $\V$-categorical compact Hausdorff
space $\V=(\V,\hom,\xi)$.

\begin{examples}
  \begin{enumerate}
  \item Our motivating example is produced by $\thU=\thU_\two$. In this case,
    the objects of the Eilenberg-Moore category for the monad $\mU$ on $\ORD$
    are precisely the ordered compact Hausdorff spaces introduced in
    \cite{Nac50}, and the homomorphisms are the monotone continuous map. We
    denote this category by $\ORDCH$. We recall that an \df{ordered compact
      Hausdorff space} $X$ is a set equipped with an order relation $\le$ and a
    compact Hausdorff topology so that
    \[
      \{(x,y)\mid x\le y\}\subseteq X\times X
    \]
    is closed with respect to the product topology. It is shown in \cite{Tho09}
    that this condition is equivalent to being an Eilenberg--Moore algebra for
    the ultrafilter monad on $\ORD$.
  \item For $\thU=\thU_{\Pp}$, we put $\METCH=\MET^\mU$ and call an object of
    $\METCH$ a \df{metric compact Hausdorff space}.
  \item Similarly, for $\thU=\thU_\ddf$, the objects of $\PROBMET^\mU$ are
    called \df{probabilistic metric compact Hausdorff spaces}. The category
    $\PROBMET^\mU$ will be represented by $\PROBMETCH$.
  \end{enumerate}
\end{examples}

\begin{proposition}
  Let $\thU_1=\utheoryone$ and $\thU_2=\utheorytwo$ be ultrafilter theories and
  $\varphi:\V_1\to\V_2$ be a lax quantale morphism compatible with $\thU_1$ and
  $\thU_2$. Then, for every $\V_1$-category $X$, the identity map on the set
  $UX$ is a $\V_2$ functor of type
  \[
    UB_\varphi(X) \longrightarrow B_\varphi U(X).
  \]
\end{proposition}
\begin{proof}
  Let $(X,a_0)$ be a $\V_1$-category. Then, since $\varphi$ is compatible with
  the ultrafilter theories $\thU_1$ and $\thU_2$, for all $\fx,\fy\in UX$ we
  have
  \begin{align*}
    \Uxitwo(\varphi a_0)(\fx,\fy)
    &=\bigvee\{\xi_2\cdot U\varphi\cdot Ua_0(\fw)\mid U\pi_1(\fw)=\fx,U\pi_2(\fw)=\fy\}\\
    &\leq \bigvee\{ \varphi\cdot\xi_1 \cdot Ua_0(\fw)\mid U\pi_1(\fw)=\fx,U\pi_2(\fw)=\fy\}\\
    &\le\varphi\left( \bigvee\{\xi_1 \cdot Ua_0(\fw)\mid U\pi_1(\fw)=\fx,U\pi_2(\fw)=\fy\}\right)\\
    &=\varphi(\Uxione a_0)(\fx,\fy);
  \end{align*}
  which proves the claim.
\end{proof}

Hence, the family of these maps defines a natural transformation
\[
  \xymatrix{\Cats{\V_1}\ar[r]^{B_\varphi}\ar[d]_{U}\drtwocell<\omit> &
    \Cats{\V_2}\ar[d]^{U} \\
    \Cats{\V_1}\ar[r]_{B_\varphi} & \Cats{\V_2}}
\]
and, together with $B_\varphi:\Cats{\V_1}\to\Cats{\V_2}$, a monad morphism (see
\cite{Pum70}) from the ultrafilter monad $\mU$ on $\Cats{\V_1}$ to the
ultrafilter monad $\mU$ on $\Cats{\V_2}$. As a result, we obtain the functor
\[
  B_\varphi:\Cats{\V_1}^{\mU} \longrightarrow\Cats{\V_2}^\mU
\]
sending $(X,a_0,\alpha)$ to $(X,\varphi a_0,\alpha)$ and making the diagram
\[
  \xymatrix{\Cats{\V_1}^\mU\ar[r]^{B_\varphi}\ar[d]_{G^\mU} &
    \Cats{\V_2}^\mU\ar[d]^{G^\mU} \\ \Cats{\V_1}\ar[r]_{B_\varphi} &
    \Cats{\V_2}}
\]
commutative. In particular, for every completely distributive quantale $\V$ and
$\xi$ given by the formula in Theorem~\ref{thm:6}, the lax quantale morphism
$p:\V\to \two$ induces the functor
\[
  B_p:\Cats{\V}^{\mU} \longrightarrow \ORDCH.
\]

\begin{examples}
  We have seen in Example~\ref{exs:functors_Ucompatible} that the lax quantale
  morphisms introduced in Example~\ref{exs:quantale_morphs}
  (\ref{functors_in_delta}) are compatible with the ultrafilter theories
  $\thU_{\Pp}$ and $\thU_\ddf$. As a consequence one has the adjoint functors
  \[
    \PROBMETCH \adjunct{B_{I_{\infty}}}{B_{P_{\infty}}} \METCH.
  \]
\end{examples}

\subsection{$\thU$-categories and $\thU$-functors}
\label{sec:categ-funct}

We have already mentioned in Section~\ref{sec:introduction} that there is a
close connection between ordered compact Hausdorff spaces and certain
topological spaces. In this subsection we recall the definition of
$\thU$-categories as enriched substitutes of topological spaces. This notion has
its roots in Barr's ``relational algebras'' \cite{Bar70}, an extensive
presentation of the theory of $(\mT,\V)$-categories (also called
$(\mT,\V)$-algebras), for a monad $\mT$ and a quantale $\V$, can be found in
\cite{HST14}.

\begin{definition}\label{def:1}
  A \df{$\thU$-category} is a pair $(X,a)$ consisting of a set $X$ and a
  $\V$-relation $a:TX\relto X$ satisfying the lax Eilenberg--Moore axioms
  $1_X\le a\cdot e_X$ and $a\cdot\Uxi a\le a\cdot m_X$.
\end{definition}

Expressed elementwise, these two conditions read as
\begin{align*}
  k&\le a(e_X(x),x) &&\text{and}& \Uxi a(\fX,\fx)\otimes a(\fx,x)\le a(m_X(\fX),x),
\end{align*}
for all $\fX\in UUX$, $\fx\in UX$ and $x\in X$.
 
\begin{definition}\label{def:2}
  A function $f:X\to Y$ between $\thU$-categories $(X,a)$ and $(Y,b)$ is a
  \df{$\thU$-functor} whenever $f\cdot a\le b\cdot Uf$.
\end{definition}
Since $f\dashv f^\circ$ in $\Rels{\V}$, this condition is equivalent to
$a\le f^\circ\cdot b\cdot Uf$, and in pointwise notation the latter inequality
becomes
\[
  a(\fx,x)\le b(Uf(\fx),f(x)),
\]
for all $\fx\in UX$, $x\in X$. The category of $\thU$-categories and
$\thU$-functors is denoted by $\Cats{\thU}$.

\begin{examples}\label{exs:U-Cats}
  \begin{enumerate}
  \item For $\V=\two$, a $\thU_\two$-category is a set $X$ equipped with a
    relation $\to:UX\relto X$ such that $e_x(x)\to x$ and, if $\fX \to \fr$ and
    $\fr\to x$, then $m_X(\fX)\to x$. It is shown in \cite{Bar70} that these are
    precisely the convergence relations induced by topologies; in fact, the main
    result of \cite{Bar70} states that $\Cats{\thU_\two}$ is isomorphic to the
    category $\TOP$ of topological spaces and continuous maps.
  \item The concept of approach space was introduced by Lowen in 1989 (see
    \cite{Low89,Low97}). It involves a set $X$ and a map
    $\delta: PX\times X\to \zerinf$, called approach distance or distance map,
    satisfying:
    \begin{enumerate}
    \item $\delta(\{x\},x)=0$,
    \item $\delta (\varnothing, x)=\infty$,
    \item $\delta (A\cup B,x)=\min\{\delta(A,x),\delta(B,x)\}$,
    \item $\delta (A^{(\varepsilon)},x)+\varepsilon \geq \delta(A,x)$, with
      $A^{(\varepsilon)}=\{x\in X\ |\ \varepsilon\geq \delta(A,x)\}$,
    \end{enumerate}
    for all $x\in X$, all $A,B \in PX$ and all $\varepsilon \in \zerinf$. A
    non-expansive map is a map $f:X \to Y$ between approach spaces $(X,\delta)$
    and $(Y,\delta ')$ subject to $\delta (A,x)\geqslant \delta '(f(A),f(x))$,
    for all $A\in PX$ and all $x\in X$.
    
    It was proved in \cite{CH03} that a ${\Pp}$-relation $a:UX\relto X$ is
    induced by an approach distance $\delta :PX\times X\to \zerinf$ if and only
    if
    \begin{align*}
      0\geqslant a(\doo{x},x)
      && \text{and}
      && U_{\xi}a(\fX, \fr) + a(\fr,x) \geqslant a(m_X(\fX),x).
    \end{align*}
    for all $x\in X$, all $\fr \in UX$ and all $\fX\in UUX$, or equivalently, if
    and only if $(X,a)$ is a $\thU_{\Pp}$-category. Moreover, non-expansive maps
    correspond precisely to $\thU_{\Pp}$-functors, so that
    $\APP\simeq\Cats{\thU_{\Pp}}$.
  \item For $\V=\Pm \simeq \zerone_{\wedge}$, $\Cats{\thU_{\Pm}}$ can be
    identified with the subcategory of $\APP$ whose objects $(X,a)$ are the
    approach spaces satisfying additionally the condition
    \[
      \max (U_{\xi}a(\fX,\fr), a(\fr,x))\geqslant a(m_X (\fX),x),
    \]
    for all $\fX\in UUX$, $\fr\in UX$ and $x\in X$.
  \item For $\V=\zerone_{\odot}$, $\Cats{\thU_{[0,1]_\odot}}$ is the category
    whose objects are structures of the type $(X,a)$ with $a:UX\relto X$
    satisfying
    \begin{align*}
      1\leqslant  a(\doo{x},x)
      && \text{and}
      && U_{\xi}a(\fX, \fr) + a(\fr,x) \geqslant 1 \implies U_{\xi}a(\fX, \fr) + a(\fr,x) \leqslant a(m_X(\fX),x)+1.
    \end{align*}
  \item For $\V=\ddf$, $\thU$-categories can be identified with probabilistic
    approach spaces. This is an example of a \emph{quantale-valued approach
      space} studied in \cite{LT16_tmp}. For the sake of simplicity we will
    represent $\Cats{\thU_\ddf}$ by $\PROBAPP$.
  \end{enumerate}
\end{examples}

Similarly to the situation for $\V$-categories, the canonical forgetful functor
$\FgtSet{\thU}:\Cats{\thU}\to\SET$ is topological (see \cite{CH03}); which
implies that the category $\Cats{\thU}$ is complete and cocomplete and
$\FgtSet{\thU}$ preserves limits and colimits. We denote the free
$\thU$-category over the one-element set $1$ by $G=(1,k)$; here $k:U1\relto 1$
is the $\V$-relation which sends the unique element of $U1\times 1$ to $k$.

There are several other functors connecting $\thU$-categories with
$\V$-categories and topological spaces. Firstly, we have a functor
$\SET^\mU\hookrightarrow\Cats{\thU}$ interpreting an Eilenberg--Moore algebra as
a lax one. Furthermore, there is a forgetful functor
$(-)_0:\Cats{\thU}\to\Cats{\V}$ sending $(X,a)$ to $(X,a_0=a\cdot e_X)$ and
leaving maps unchanged. We notice that the diagram
\[
  \xymatrix{\SET^\mU\ar[r]\ar[d]_{G^\mU} & \Cats{\thU}\ar[d]^{(-)_0} \\
    \SET\ar[r]_-{\text{discrete}} & \Cats{\V}}
\]
commutes. Furthermore, by \cite[Section 4]{Hof07}, we have:

\begin{proposition}\label{d:prop:5}
  Assume that $\varphi:\V_1\to\V_2$ is is a lax quantale morphism compatible
  with the ultrafilter theories $\thU_1=\utheoryone$ and
  $\thU_2=\utheorytwo$. Then
  \begin{align*}
    (X,a)\longmapsto (X,\varphi a) &&\text{and}&& f \longmapsto f 
  \end{align*}
  define a functor $B_\varphi:\Cats{\thU_1}\to\Cats{\thU_2}$.
\end{proposition}

By Proposition~\ref{d:prop:5}, for every ultrafilter theory $\thU=\utheory$, the
canonical map $i:\two\to\V$ (see Subsection~\ref{ssec:quant-enrich-categ})
induces the functor
\[
  B_i:\TOP \longrightarrow \Cats{\thU}.
\]
interpreting a topological space $X$ as the $\thU$-category with structure
\[
  (\fx,x) \longmapsto \begin{cases}
    k & \text{if }\fx\to x\\
    \bot & \text{else,}
  \end{cases}
\]
for $\fx \in UX$ and $x\in X$. If, moreover, the right adjoint $p:\V\to\two$ of
$i$ is compatible with $\thU$ and $\thU_\two$ (see Lemmas~\ref{d:lem:3} and
\ref{d:lem:4}), then $p$ defines a functor
\[
  B_p:\Cats{\thU} \longrightarrow \TOP
\]
which is right adjoint to $B_i$. Here $B_p$ sends an $\thU$-category $(X,a)$ to
the topological space $X$ with convergence
\[
  UX\times X\xrightarrow{\quad a\quad}\V\xrightarrow{\quad p\quad}\two;
\]
that is, for $\fx \in UX$ and $x\in X$, $\fx\to x$ if and only if $k\le a(\fx,x)$.

\begin{examples} \label{functors for U-cats} The adjoint lax quantale morphisms
  $I_\infty \dashv P_\infty$ (see Example~\ref{exs:quantale_morphs}) are
  compatible with the ultrafilter theories $\thU_{\Pp}$ and
  $\thU_{\ddf}$. Therefore they induce the adjoint functors
  \[
    \PROBAPP \adjunct{B_{I_\infty}}{B_{P_\infty}} \APP.
  \]
\end{examples}

\subsection{Comparison with $\thU$-categories}

It is shown in \cite{Tho09} that there is a canonical functor
\[
  K:(\Cats{\V})^\mU \longrightarrow \Cats{\thU}
\]
which associates to each $X=(X,a_0,\alpha)$ in $(\Cats{\V})^\mU$ the
$\thU$-category $KX=(X,a)$ where $a=a_0\cdot\alpha$. Note that
$(a_0\cdot\alpha)_0=a_0$, that is, the diagram
\begin{equation}\label{d:eq:2}
  \xymatrix{(\Cats{\V})^\mU\ar[r]^-{K}\ar[dr]_{G^\mU} &
    \Cats{\thU}\ar[d]^{(-)_0}\\ & \Cats{\V}}
\end{equation}
commutes. Applying $K$ to $\V=(\V,\hom,\xi)$ produces the $\thU$-category
$\V=(\V,\hom_\xi)$ where
\[
  \hom_\xi:U\V\times\V \longrightarrow \V,\,(\fv,v)\longmapsto\hom(\xi(\fv),v).
\]

\begin{example}
  \begin{enumerate}
  \item For $\thU=\thU_\two$, one obtains the commutative diagram
    \[
      \xymatrix{\ORDCH\ar[r]^-{K}\ar[dr]_{G^\mU} & \TOP\ar[d]^{(-)_0}\\ & \ORD}
    \]
    Here every ordered compact Hausdorff space maps to a weakly sober, locally
    compact and stable topological space; assuming also the T0-axiom, these
    spaces are called \df{stably compact} (see \cite{GHK+03}). It is also shown
    in \cite{GHK+03} that the full subcategory of $\ORDCH$ defined by the
    separated orders is isomorphic to the category $\STCOMP$ of stably compact
    topological spaces and spectral maps. We also note that the space $K\two$ is the
    Sierpi\'nski space $\two=\{0,1\}$ with $\{1\}$ closed.
  \item When we consider the ultrafilter theory $\thU=\thU_{\Pp}$, the diagram
    \eqref{d:eq:2} becomes
    \[
      \xymatrix{\METCH\ar[r]^-{K}\ar[dr]_{G^\mU} & \APP\ar[d]^{(-)_0}\\ & \MET.}
    \]
    Here the space $K\Pp$ coincides with the ``Sierpi\'nski approach space'' of
    \cite[Example~1.8.33~(2)]{Low97}. Similarly to the topological case, it is
    shown in \cite{GH13} that separated metric compact Hausdorff spaces
    correspond precisely to stably compact approach spaces.
  \item For $\thU=\thU_{\ddf}$, we obtain the diagram
    \[
      \xymatrix{\PROBMETCH\ar[r]^-{K}\ar[dr]_{G^\mU} & \PROBAPP\ar[d]^{(-)_0}\\
        & \PROBMET.}
    \]
  \end{enumerate}
\end{example}

The functor $K:(\Cats{\V})^\mU\to\Cats{\thU}$ is right adjoint, its left adjoint
assigns to every $\thU$-category the $\V$-categorical compact Hausdorff space
$(UX,\hat{a},m_X)$. Regarding this construction, we recall here from
\cite{CH09}:

\begin{lemma}\label{lem:2}
  For every $\thU$-category $(X,a)$, $\hat{a}:=\Uxi a\cdot m_X^\circ$ is a
  $\V$-category structure on $UX$.
\end{lemma}

We give now an alternative characterisation of the compatibility between the
convergence and the $\V$-categorical structure of an Eilenberg--Moore algebra
$(X,a_0,\alpha)$ in $\Cats{\V}^\mU$, which resembles the classical condition
stating that ``the order relation is closed in the product space''.

\begin{proposition}\label{prop:1}
  For a $\V$-category $(X,a_0)$ and a $\mU$-algebra $(X,\alpha)$ with the same
  underlying set $X$, the following assertions are equivalent.
  \begin{tfae}
  \item $\alpha:U(X,a_0)\to(X,a_0)$ is a $\V$-functor.
  \item $a_0:(X,\alpha)\times(X,\alpha)\to(\V,\hom_\xi)$ is an $\thU$-functor.
  \end{tfae}
\end{proposition}
\begin{proof}
  The first assertions is equivalent to
  \[
    \forall\fx,\fy\in UX\,.\,\Uxi a_0(\fx,\fy)\le a_0(\alpha(\fx),\alpha(\fy)),
  \]
  and, since
  $\Uxi a_0(\fx,\fy)=\bigvee\{\xi\cdot Ua_0(\fw)\mid\fw\in U(X\times
  X),U\pi_1(\fw)=\fx,U\pi_2(\fw)=\fy\}$, this is equivalent to
  \[
    \forall\fx,\fy\in UX,\forall\fw\in U(X\times X)\,.\,
    ((U\pi_1(\fw)=\fx\,\&\,U\pi_2(\fw)=\fy) \implies(\xi\cdot Ua_0(\fw)\le
    a_0(\alpha(\fx),\alpha(\fy)))).
  \]
  On the other hand, the second statement translates to
  \[
    \forall\fw\in U(X\times X), \forall\fx,\fy\in UX\,.\,
    ((U\pi_1(\fw)=\fx\,\&\,U\pi_2(\fw)=\fy) \implies (k\le\hom(\xi\cdot
    Ua_0(\fw), a_0(\alpha(\fx),\alpha(\fy))))),
  \]
  which proves the equivalence.
\end{proof}

From Proposition~\ref{prop:1} we conclude immediately:

\begin{lemma}
  Let $\thU$ be an ultrafilter theory so that $p:\V\to\two$ is compatible with
  $\thU$ and $\thU_\two$, and $(X,a_0,\alpha)$ be a $\V$-categorical compact
  Hausdorff space. Then, for all $x\in X$ and $u\in\V$, the closed balls
  \begin{align*}
    \{y\in X\mid a_0(x,y)\ge u\}
    &&\text{and}
    &&\{y\in X\mid a_0(y,x)\ge u\}
  \end{align*}
  with center $x$ and radius $u$ are closed with respect to the compact
  Hausdorff topology.
\end{lemma}
\begin{proof}
  Applying the forgetful functor $B_p:\Cats{\thU}\to\TOP$, we obtain that
  $a_0:(X,\alpha)\times(X,\alpha)\to B_p(\V,\hom_\xi)$ is continuous. For every
  $x\in X$, consider the continuous maps
  \begin{align*}
    X\xrightarrow{\,\langle 1_X,x\,\rangle}X\times X\xrightarrow{\,a_0\,}\V
    &&\text{and}
    && X\xrightarrow{\,\langle x,1_X\,\rangle}X\times X\xrightarrow{\,a_0\,}\V.
  \end{align*}
  Then, for every $u\in\V$, the closed balls with center $x$ and radius $u$ are
  the preimages of the closed set $\upc u$.
\end{proof}

\subsection{Convergences from $\V$-categories}
\label{sec:convergences-from-v}

We recall from Corollary~\ref{d:cor:4} that, for every quantale $\V=\quantale$
where $k$ is $\vee$-irreducible, we have the functor
\[
  T_\V:\Cats{\V}\longrightarrow\TOP
\]
sending a $\V$-category $(X,a_0)$ to $X$ equipped with the L-closure of
$(X,a_0)$.  We investigate now connections between $\V$-categorical and
topological properties. Note that, if $k$ is terminal in the quantale $\V$, then
the projection maps $\pi_1:X\times Y\to X$ and $\pi_2:X\times Y\to Y$ are
$\V$-functors
\begin{align*}
  \pi_1:(X,a_0)\otimes (Y,b_0) \longrightarrow (X,a_0)
  &&\text{and}
  && \pi_2:(X,a_0)\otimes (Y,b_0) \longrightarrow (Y,b_0),
\end{align*}
for all $\V$-categories $(X,a_0)$ and $(Y,b_0)$. Therefore, with the same proof
as for \cite[Corollary 5.8]{HT10}, we obtain:

\begin{lemma}\label{d:lem:2}
  Let $\V$ be a quantale where $k$ is terminal and $(X,a_0)$ be
  $\V$-category. For all $x,y\in X$,
  \[
    x\simeq y\iff (x,y)\in\overline{\Delta}\text{ in $(X,a_0)\otimes(X,a_0)$.}
  \]
  Hence, $(X,a_0)$ is separated if and only if $\Delta$ is closed in
  $(X,a_0)\otimes(X,a_0)$.
\end{lemma}

In the sequel we will often require that the functor $T_\V$ is monoidal.
\begin{definition}
  The functor $T_\V:\Cats{\V}\to\TOP$ is \df{monoidal} if, for all
  $\V$-categories $(X,a_0)$ and $(Y,b_0)$, the identity map on $X\times Y$ is
  continuous of type
  \[
    T_\V(X,a_0)\times T_\V(Y,b_0) \longrightarrow T_\V((X,a_0)\otimes(Y,b_0)).
  \]
\end{definition}

\begin{proposition}\label{d:prop:3}
  Let $\V=\quantale$ be a completely distributive quantale where $k$ is
  $\vee$-irreducible and terminal in $\V$. We consider the ultrafilter theory
  $\thU=\utheory$ where $\xi:U\V\to\V$ is as in Theorem~\ref{thm:6}. Then the
  following assertions hold.
  \begin{enumerate}
  \item\label{d:item:1} The identity map on $\V$ is continuous of type
    $T_\V(\V,\hom)\to B_p(\V,\hom_\xi)$.
  \item\label{d:item:2} Assume that $T_\V$ is monoidal. Then, for every
    $\V$-category $(X,a_0)$, the topological space $T_\V(X,a_0)$ is Hausdorff if
    and only if $(X,a_0)$ is separated.
  \end{enumerate}
\end{proposition}
\begin{proof}
  To see (\ref{d:item:1}), we note that an ultrafilter $\fx$ converges to
  $x$ in $T_\V(\V,\hom)$ if and only if, for all $A\in\fx$,
  $x\in\overline{A}$; that is,
  \[
    k\leq \bigvee_{z\in A}\hom(x,z)\otimes\hom(z,x).
  \]
  On the other hand, $\fx\to x$ in $B_p(\V,\hom_\xi)$ is equivalent to
  \[
    \forall A\in\fx\,.\,(\bigwedge A\leq x).
  \]
  Assume $\fx\to x$ in $T_\V(\V,\hom)$. For every $A\in\fx$, we calculate
  \[
    x=\hom(k,x)\geq\hom(\bigvee_{z\in A}\hom(x,z)\otimes\hom(z,x),x)
    =\bigwedge_{z\in A}\hom(\hom(x,z)\otimes\hom(z,x),x).
  \]
  Since $k$ is terminal in $\V$,
  \[
    z\otimes\hom(x,z)\otimes\hom(z,x)\le x\otimes \hom(x,z)\le x;
  \]
  we obtain
  \[
    \bigwedge_{z\in A}z\leq\bigwedge_{z\in A}\hom(\hom(x,z)\otimes\hom(z,x),x)
    \leq x.
  \]
  Regarding (\ref{d:item:2}), by Lemma~\ref{d:lem:2}, a $\V$-category $(X,a_0)$
  is separated if and only
  \[
    \Delta_X\subseteq X\times X
  \]
  is closed in $T_\V((X,a_0)\otimes(X,a_0))$. Hence, since $T_\V$ is monoidal,
  the assertion follows.
\end{proof}

\begin{definition}
  Let $\V=\quantale$ be a quantale where $k$ is $\vee$-irreducible. A
  $\V$-category $X$ is called \df{compact} whenever the topological space
  $T_\V(X)$ is compact. The full subcategory of $\Cats{\V}$ defined by all
  compact separated $\V$-categories is denoted by $\Cats{\V}_{\sep,\comp}$.
\end{definition}

\begin{theorem}
  Let $\V$ be a completely distributive quantale where $k$ is $\vee$-irreducible
  and terminal in $\V$ and assume that $T_\V:\Cats{\V}\to\TOP$ is monoidal. Let
  $\thU=\utheory$ be the ultrafilter theory with $\xi:U\V\to\V$ as in
  Theorem~\ref{thm:6}. Then the functor
  $T_\V:\Cats{\V}_{\sep,\comp}\to\COMPHAUS$ lifts to a functor
  $T_\thU:\Cats{\V}_{\sep,\comp}\to(\Cats{\V})^\mU$ which commutes with the
  canonical forgetful functors
  \[
    \xymatrix{ & \Cats{\V}\\
      \Cats{\V}_{\sep,\comp}\ar[r]^-{T_\thU}\ar[dr]_{T_\V}\ar[ur] & (\Cats{\V})^\mU\ar[u]\ar[d]\\
      & \COMPHAUS}
  \]
  to $\Cats{\V}$ and $\COMPHAUS$.
\end{theorem}
\begin{proof}
  Since the composite
  \[
    T_\V(X,a_0)\times T_\V(X,a_0)
    \xrightarrow{\text{can}}T_\V((X,a_0)^\op\otimes(X,a_0)) \xrightarrow{T_\V
      a_0}T_\V(\V,\hom) \longrightarrow B_p(\V,\hom_\xi)
  \]
  is continuous, every separated compact $\V$-category becomes a
  $\V$-categorical compact Hausdorff space when equipped with the topology of
  $T_\V(X)$.
\end{proof}

\begin{proposition}
  Let $\V=\quantale$ be a quantale where $k$ is approximated. Then $T_\V$ is
  monoidal.
\end{proposition}
\begin{proof}
  Under the condition that $k$ is approximated, the topology of a $\V$-category
  $(X,a_0)$ is generated by the ``symmetric open balls''
  \[
    \SOB(x,u)=\{y\in X\mid u\ll a_0(x,y)\text{ and }u\ll a_0(y,x)\},
  \]
  for $x\in X$ and $u\ll k$ (see Proposition~\ref{d:prop:4} and \cite[Remark
  4.21]{HR13}). Let now $(X,a_0)$ and $(Y,b_0)$ be $\V$-categories,
  $(x,y)\in X\times Y$ and $u\ll k$. By Proposition~\ref{d:prop:4}, there is
  some $v\ll k$ with $u\leq v\otimes v$. Then
  $\SOB(x,v)\times\SOB(y,v)\subseteq\SOB((x,y),u)$; which proves the claim.
\end{proof}

\begin{examples}
  By the results of this subsection, every compact separated (probabilistic)
  metric space is a (probabilistic) metric compact Hausdorff space.
\end{examples}

\section{Completeness from compactness}

The central topic of this section is Cauchy completeness for $\thU$-categories,
a notion defined in terms of adjoint $\thU$-distributors (see \cite{CH09}). As
much as possible we avoid assuming that the ultrafilter theory $\thU$ is strict;
as a consequence, we cannot assume associativity of the composition of
$\thU$-distributors. The lack of associativity forces us to be more careful in
our treatment of adjunctions; in particular, adjoints need not be unique. We
show, under some conditions on the quantale $\V$, that the corresponding
$\thU$-category of a $\V$-categorical compact Hausdorff space is Cauchy
complete. Moreover, for strict theories $\thU$, we prove that the forgetful
functor $\Cats{\thU}\to\Cats{\V}$ preserves Cauchy-completeness. Combining both
results shows that the underlying $\V$-category of a $\V$-categorical compact
Hausdorff space is Cauchy complete. In the last subsection we go a step further
and study codirected completeness for $\V$-categories.

\subsection{$\thU$-distributors}
\label{sec:Udistr}

A $\V$-relation of the form $\varphi:UX\relto Y$ is called a
\df{$\thU$-relation} and think of $\varphi$ as an arrow from $X$ to $Y$, and
write $\varphi:X\krelto Y$. Composition is given by \df{Kleisli composition}:
\begin{gather*}
  \psi\circ \varphi:=\psi\cdot U_{\!_\xi} \varphi\cdot m_{X}^\circ,\\
  \xymatrix{UX\ar[dd]|{\object@{|}}^\varphi & UY\ar[dd]|{\object@{|}}^\psi & & & UX\ar[r]|{\object@{|}}^{m_{X}^\circ} \ar[rdd]|{\object@{|}}_{\psi\circ \varphi} & UUX\ar[d]|{\object@{|}}^{U_{\!_\xi}\varphi}\\
    && \ar@{|->}[r] &  && UY\ar[d]|{\object@{|}}^\psi\\
    Y & Z &&&& Z}
\end{gather*}
for all $\varphi:X\krelto Y$ and $\psi:Y\krelto Z$. One easily verifies
\begin{align*}
  \varphi\kleisli e_{X}^\circ&=\varphi\cdot Ue_{X}^\circ\cdot m_{X}^\circ=\varphi
                               \intertext{and}
                               e_{X}^\circ\kleisli \varphi &= e_{X}^\circ\cdot U_{\!_\xi}\varphi\cdot m_{X}^\circ\ge \varphi\cdot e_{UX}^\circ\cdot m_{X}^\circ=\varphi,
\end{align*}
for all $\thU$-relations $\varphi:X\krelto Y$; that is, $e_{X}^\circ$ is a lax
identity for the Kleisli composition. Moreover:

\begin{theorem}\label{thm:2}
  For composable $\thU$-relations we have
  \[
    \varphi\kleisli(\psi\kleisli\gamma)\ge (\varphi\kleisli\psi)\kleisli\gamma,
  \]
  with equality if $\thU$ is a strict theory.
\end{theorem}
\begin{proof}
  See \cite[Subsection 2.1]{Hof05}.
\end{proof}

\begin{remark}
  In the language of $\thU$-relations, an $\thU$-category $(X,a)$ consists of a
  set $X$ and an $\thU$-relation $a:X\krelto X$ satisfying
  \begin{align*}
    e_X^\circ\le a &&\text{and}&& a\kleisli a\le a.
  \end{align*}
\end{remark}

\begin{definition}\label{def:3}
  A $\thU$-relation $\varphi:X\krelto Y$ between $\thU$-categories $X=(X,a)$ and
  $Y=(Y,b)$ is a \df{$\thU$-distributor}, written as $\varphi:X\kmodto Y$,
  whenever $\varphi\kleisli a\le\varphi$ and $b\kleisli\varphi\le\varphi$.  In
  pointwise notation $\varphi:X\krelto Y$ is $\thU$-distributor if, for all
  $\fr \in UX$, all $\fX \in UUX$, all $y\in Y$ and all $\fy \in UY$,
  \begin{align*}
    U_{\xi}a(\fX, \fr)\otimes\varphi(\fr,y)\le \varphi(m_X(\fX),y)
    && \text{and}
    && U_{\xi}\varphi(\fX, \fy)\otimes b(\fy,y)\le \varphi(m_X(\fX), y).
  \end{align*}
\end{definition}

In other words, an $\thU$-distributor $\varphi:X\kmodto Y$ comes with a
\emph{right action} of the $\thU$-relation $a$ and a \emph{left action} of
$b$. This perspective motivates the designations \emph{bimodule} or
\emph{module} used by some authors. Note that we always have
$\varphi\kleisli a\ge\varphi$ and $b\kleisli \varphi\ge \varphi$, so that the
$\thU$-distributor conditions above are in fact equalities which make the
$\thU$-structures identities for the composition of $\thU$-distributors.

\begin{remark}\label{rem:1}
  In general, $\thU$-distributors do not compose. However, this property is
  guaranteed by assuming that the ultrafilter theory is strict.
\end{remark}

The following result establishes a connection between $\thU$-distributors and
$\thU$-functors and generalises slightly \cite[Theorem~4.3]{CH09}.

\begin{theorem}\label{thm:3}
  Let $(X,a)$ and $(Y,b)$ be $\thU$-categories, and $\varphi:UX\relto Y$ be a
  $\V$-relation. Then the following assertions are equivalent.
  \begin{tfae}
  \item The $\V$-relation $\varphi$ is an $\thU$-distributor
    $\varphi:X\kmodto Y$.
  \item $\varphi:(UX,\hat{a})^\op\times(Y,1_Y)\to(\V,\hom)$ is a $\V$-functor
    and $\varphi:(UX,m_X)\times(Y,b)\to(\V,\hom_\xi)$ is an $\thU$-functor.
  \end{tfae}
\end{theorem}
\begin{proof}
  First note that $\varphi$ is an $\thU$-distributor if and only if
  \begin{align*}
    \varphi\cdot\hat{a}\le \varphi
    &&\text{and}
    && b\cdot\Uxi\varphi\le\varphi\cdot m_X.
  \end{align*}
  The first inequality above means precisely that, for all $y\in Y$ and all
  $\fx,\fy\in UX$,
  \[
    \varphi(\fx,y)\otimes\hat{a}(\fy,\fx)\le\varphi(\fy,y),
  \]
  which in turn is equivalent to
  \[
    \hat{a}(\fy,\fx)\le\hom(\varphi(\fx,y),\varphi(\fy,y)).
  \]
  Consequently, $\varphi\cdot\hat{a}\le \varphi$ if and only if, for all
  $y\in Y$,
  \[
    \varphi(-,y):(UX,\hat{a})^\op\to(\V,\hom)
  \]
  is a $\V$-functor; which is the case if and only if
  $\varphi:(UX,\hat{a})^\op\times(Y,1_Y)\to(\V,\hom)$ is a $\V$-functor.
	
  Secondly, $b\cdot\Uxi\varphi\le\varphi\cdot m_X$ if and only if, for all
  $\fX\in UUX$, $\fy\in UY$ and $y\in Y$,
  \[
    b(\fy,y)\otimes\Uxi\varphi(\fX,\fy)\le \varphi(m_X(\fX),y),
  \]
  and this inequality is equivalent to
  \[
    \bigvee\{b(\fy,y)\otimes\xi U\varphi(\fW)\mid \fW\in U(UX\times
    Y),U\pi_1(\fW)=\fX,U\pi_2(\fW)=\fy\} \le \varphi(m_X(\fX),y).
  \]
  The latter holds if and only if, for all $\fW\in U(UX\times Y)$, $\fx\in UX$
  and $y\in Y$ with $m_X(U\pi_1(\fW))=\fx$,
  \[
    b(U\pi_2(\fW),y)\le\hom(\xi U\varphi(\fW),\varphi(\fx,y)).
  \]
  Hence, $b\cdot\Uxi\varphi\le\varphi\cdot m_X$ is equivalent to
  $\varphi:(UX,m_X)\times(Y,b)\to(\V,\hom_\xi)$ being an $\thU$-functor.
\end{proof}

In the sequel we will consider in particular $\thU$-distributors with domain or
codomain $G$. For an $\thU$-category $X=(X,a)$, an $\thU$-relation
$\varphi:1\krelto X$ is an $\thU$-distributor $\varphi: G\kmodto X$ if and only
if, for all $x\in X$ and all $\fr \in UX$,
\[
  U_{\xi} \varphi(\fr)\otimes a(\fr,x)\le \varphi(x).
\]
Similarly, an $\thU$-relation $\psi:X\krelto 1$ is an $\thU$-distributor
$\psi:X\kmodto G$ if and only if, for all $\fr \in UX$ and all $\fX \in UUX$,
\begin{align*}
  U_{\xi} a(\fX,\fr)\otimes \psi(\fr)\le \psi(m_X(\fX))
  && \text{and}
  && U_{\xi}\psi(\fX)\le \psi(m_X(\fX)).
\end{align*}

Let $(X,a)$ and $(Y,b)$ be $\thU$-categories. Each map $f:X\to Y$ induces
$\thU$-relations
\begin{align*}
  f_\Tast=b\cdot Uf:X\krelto Y &&\text{and}&& f^\Tast=f^\circ\cdot b:Y\krelto X; 
\end{align*}
moreover, one has $b\kleisli f_\Tast\le f_\Tast$ and
$f^\Tast\kleisli b\le f^\Tast$. These $\thU$-relations are actually
$\thU$-distributors precisely when $f$ is an $\thU$-functor.

\begin{lemma}\label{lem:1}
  The following are equivalent, for $\thU$-categories $(X,a)$ and $(Y,b)$ and a
  map $f:X\to Y$.
  \begin{tfae}
  \item $f$ is an $\thU$-functor $f:(X,a)\to(Y,b)$.
  \item $f_\Tast$ is an $\thU$-distributor, that is,
    $f_\Tast\kleisli a\le f_\Tast$.
  \item $f^\Tast$ is an $\thU$-distributor, that is, $a\kleisli f^\Tast\le f\Tast$.
  \end{tfae}
\end{lemma}
\begin{proof}
  See \cite[Subsection 3.6]{CH09}.
\end{proof}

\begin{lemma}
  Let $f:A\to X$ and $g:Y\to B$ be $\thU$-functors and $\varphi:X\kmodto Y$ be an
  $\thU$-distributor. Then
  \begin{align*}
    \varphi\kleisli f_\Tast=\varphi\cdot Uf
    &&\text{and}
    && g^\Tast\kleisli\varphi=g^\circ\cdot\varphi
  \end{align*}
  are $\thU$-distributors.
\end{lemma}
\begin{proof}
  See \cite[Proposition 3.6]{CH09}.
\end{proof}

Similarly to the case of $\V$-categories, the local order of $\Rels{\V}$ allows
us to consider $\Cats{\thU}$ as an ordered category: for $\thU$-functors
$f,g:X\to Y$,
\begin{align*}
  f\le g  & \iff f^\Tast \le g^\Tast \iff g_\Tast \le f_\Tast\\
          & \iff f^*\le g^*.
\end{align*}
In particular, every $\thU$-category $X$ has an underlying order where $x\le y$
whenever $x^\Tast\le y^\Tast$, for all $x,y\in X$; which in turn is equivalent
to $k\le a(\doo{x},y)$. This construction defines a functor
\[
  \FgtOrd{\thU}:\Cats{\thU} \longrightarrow\ORD,
\]
and the diagrams
\[
  \xymatrix{\Cats{\thU}\ar[d]_{B_p}\ar[r]^{(-)_0}\ar[dr]|{\FgtOrd{\thU}} & \Cats{\V}\ar[d]^{B_p}\\
    \TOP\ar[r]_{(-)_0} & \ORD}
\]
commute. A $\thU$-category $(X,a)$ is \df{separated} (see \cite{HT10}) whenever
the underlying ordered set $\FgtOrd{\thU}(X,a)$ is separated. We note that
$(-)_0:\Cats{\thU}\to\Cats{\V}$ sends separated $\thU$-categories to separated
$\V$-categories.

\subsection{Adjoint $\thU$-distributors}
\label{sec:adjoint-distributors}

In this subsection we study the important notion of adjoint
$\thU$-distributor. We employ here the usual definition of adjunction in an
ordered category; however, some extra caution is needed since
$\thU$-distributors in general do not compose.

\begin{definition}
  Let $X=(X,a)$ and $Y=(Y,b)$ be $\thU$-categories. A pair of $\thU$-distributors
  $\varphi:X\kmodto Y$ and $\psi:Y\kmodto X$ form an \df{adjunction}, denoted as
  $\varphi\dashv\psi$, whenever their composites, $\varphi \kleisli\psi$ and
  $\psi\kleisli\varphi$, are $\thU$-distributors and $a\le\psi\kleisli\varphi$
  and $\varphi\kleisli\psi\le b$.
\end{definition}

We hasten to remark that $f_\Tast\dashv f^\Tast$, for every $\thU$-functor
$f:(X,a)\to(Y,b)$. In fact, by \cite[Proposition 3.6 (2), p.~188]{CH09},
$f^\Tast \kleisli f_\Tast $ and $f_\Tast \kleisli f^\Tast$ are
$\thU$-distributors and
\[
  f_\Tast\kleisli f^\Tast=b\cdot Uf\cdot Uf^\circ\cdot\Uxi b\cdot m_Y^\circ\le
  b\cdot\Uxi b\cdot m_Y^\circ=b
\]
and
\[
  f^\Tast\kleisli f_\Tast =f^\circ \cdot b\cdot\Uxi b\cdot UUf\cdot m_X^\circ
  =f^\circ \cdot b\cdot\Uxi b\cdot m_Y^\circ\cdot Uf =f^\circ \cdot b\cdot Uf
  \ge f^\circ\cdot f\cdot a\ge a.
\]
Similarly to the nomenclature for $\V$-categories, we call an $\thU$-functor
$f:(X,a)\to(Y,b)$ \df{fully faithful} whenever $f^\Tast\kleisli f_\Tast=a$, and
\df{fully dense} whenever $f_\Tast\kleisli f^\Tast=b$.

In general, we are not able to prove unicity of left adjoints since composition
of $\thU$-distributors does not need to be associative. However, we can still
prove that right adjoints are unique:

\begin{proposition}\label{prop:2}
  Let $\varphi:X\kmodto Y$, $\psi:Y\kmodto X$ and $\psi':Y\kmodto X$ be
  $\thU$-distributors with $\varphi\dashv\psi$ and $\varphi\dashv\psi'$. Then
  $\psi=\psi'$.
\end{proposition}
\begin{proof}
  From $a\le\psi\kleisli\varphi$ we get
  $\psi'=a\kleisli\psi' \le(\psi\kleisli\varphi)\kleisli\psi'
  \le\psi\kleisli(\varphi\kleisli\psi') \le\psi\kleisli b=\psi$. Similarly,
  $\psi\le\psi'$, and we conclude that $\psi=\psi'$.
\end{proof}

We turn now our attention to $\thU$-distributors with domain or codomain $G$.

\begin{lemma}
  Let $X=(X,a)$ be an $\thU$-category and $\varphi:G\kmodto X$ and
  $\psi:X\kmodto G$ be $\thU$-distributors. Then the composites
  $\varphi\kleisli\psi$ and $\psi\kleisli\varphi$ are $\thU$-distributors.
\end{lemma}
\begin{proof}
  Clearly, $\psi\kleisli\varphi:G\kmodto G$ is an $\thU$-distributor. To prove
  that $\varphi\kleisli \psi$ is indeed an $\thU$-distributor of type
  $X\kmodto X$, we verify first that
  \[
    \varphi \kleisli\psi=\varphi\cdot e_1\cdot e^\circ_1 \cdot U\psi\cdot
    m^\circ_X = \varphi\cdot e_1\cdot (e^\circ_1 \kleisli\psi)= \varphi \cdot
    e_1 \cdot \psi.
  \]
  Therefore
  \[
    a\kleisli (\varphi \kleisli\psi)=a\cdot U\varphi\cdot Ue_1\cdot U\psi\cdot
    m^\circ_X \le a\cdot U\varphi\cdot m^\circ_1\cdot U\psi\cdot
    m^\circ_X=(a\kleisli \varphi)\cdot U\psi \cdot m^\circ_X= \varphi
    \kleisli\psi,
  \]
  and
  \[
    (\varphi \kleisli\psi)\kleisli a\le \varphi\kleisli(\psi \kleisli a)=\varphi
    \kleisli\psi.\qedhere
  \]
\end{proof}

Therefore, when studying adjunctions of the form
\[
  X\adjunctkmod{\varphi}{\psi}G,
\]
we do not need to worry about the composites $\varphi \kleisli\psi$ and
$\psi\kleisli\varphi$. Elementwise, $\varphi\dashv\psi$ translates to
\begin{align*}
  k\le\bigvee_{\fz\in UX}\psi(\fz)\otimes \xi U\varphi(\fz)
  &&\text{and}
  && \psi(\fx)\otimes\varphi(x)\le a(\fx,x),
\end{align*}
for all $\fx\in UX$ and $x\in X$. We also point out that
\begin{itemize}
\item A map $\varphi:X\to\V$ (seen as an $\thU$-relation $\varphi:G\krelto X$)
  is an $\thU$-distributor $\varphi:G\kmodto X$ if and only if
  $a\kleisli\varphi\le\varphi$ if and only if $\varphi:X\to\V$ is a
  $\thU$-functor (see Theorem~\ref{thm:3}) if and only if
  \[
    \varphi(x)=\bigvee_{\fx\in UX}a(\fx,x)\otimes\xi U\varphi(\fx).
  \]
\item A $\thU$-relation $\psi:X\krelto G$ is an $\thU$-distributor
  $\psi:X\kmodto G$ if and only if $\psi\kleisli a\le\psi$ and
  $e_1^\circ\cdot\Uxi\psi\cdot m_X^\circ\le\psi$.
\end{itemize}

\begin{proposition}\label{prop:3}
  Let $\psi:X\kmodto G$, $\varphi:G\kmodto X$ and $\varphi':G\kmodto X$ be
  $\thU$-distributors with $\varphi\dashv\psi$ and $\varphi'\dashv\psi$. Then
  $\varphi=\varphi'$.
\end{proposition}
\begin{proof}
  We calculate
  \[
    \varphi'(x)\le\bigvee_{\fz\in UX}\varphi'(x)\otimes\psi(\fz)\otimes \xi
    U\varphi(\fz) \le \bigvee_{\fz\in UX}a(\fx,x)\otimes \xi
    U\varphi(\fz)=\varphi(x).\qedhere
  \]
\end{proof}

\subsection{Cauchy complete $\thU$-categories}
\label{sec:cauchy-compl-categ}

With the notion of adjunction of $\thU$-distributors at our disposal, we come
now to the concept of Cauchy completeness (called Lawvere completeness in
\cite{CH09}).

\begin{definition}\label{def:5}
  A $\thU$-category $X=(X,a)$ is called \df{Cauchy complete} whenever every
  adjunction
  \[
    X\adjunctkmod{\varphi}{\psi}G,
  \]
  of $\thU$-distributors is of the form $x_\Tast\dashv x^\Tast$, for some
  $x\in X$.
\end{definition}

Note that $x_\Tast = a(\doo{x},-)$ and that $x^\Tast= a(-,x)$, so that
$x_\Tast\dashv x^\Tast$ means, for all $\fx\in UX$ and $x'\in X$,
\begin{align*}
  k\le\bigvee_{\fz\in UX}a(\fz,x)\otimes U_{\xi}x_\Tast(\fz)
  &&\text{and}
  && a(\fx,x)\otimes a(\doo{x},x')\le a(\fx,x').
\end{align*}

\begin{examples}
  Various examples of Cauchy complete $\thU$-categories are described in
  \cite{CH09}, we sketch here the principal facts.
  \begin{enumerate}
  \item We have already seen that $\TOP\simeq\Cats{\thU_\two}$. In this context,
    a $\thU_\two$-distributor $\varphi : (X,a)\kmodto (Y,b)$ is a relation
    $\varphi:UX\relto Y$ that satisfies, for all $y\in Y$, $\fy \in UY$,
    $\fr \in UX$ and $\fX \in UUX$,
    \begin{align*}
      \fX\to \fr \;\&\; \varphi (\fr, y) \implies
      \varphi(m_X(\fX),y)
      &&\text{and}
      && U_{\xi}\varphi(\fX,\fy)\;
         \&\; \fy \to y \implies \varphi(m_X(\fX),y).
    \end{align*}
    In particular, $\thU_\two$-distributors of the form $\varphi: G\kmodto X$
    can be identified with relations $\varphi:1\relto X$ satisfying
    \[
      \forall x\in X\, \forall \fr \in UX\,.\, (U_{\xi}\varphi(\fr)\; \&\; \fr
      \to x) \implies \varphi(x),
    \]
    and a relation $\psi:UX\relto 1$ is a $\thU_\two$-distributor
    $\psi: X\kmodto G$ if and only if
    \begin{align*}
      (\fX \to \fr\;\&\;\psi(\fr)) \le \psi(m_X(\fX))
      &&\text{and}
      &&U_{\xi}\psi(\fX)\le \psi(m_X(\fX)),
    \end{align*}
    for all $\fX\in UUX$ and $\fr\in UX$. Using Theorem~\ref{thm:3}, a
    $\thU_\two$-distributor $\varphi:G\kmodto X$ can be also seen as a
    continuous map $X\to\two$ into the Sierpi\'nski space, which in turn can be
    interpreted as a closed subset $A\subseteq X$. A $\thU$-distributor
    $\psi: X\kmodto G$ is a map $UX\to \two$ which is continuous with respect to
    the Zariski closure on $UX$ (an ultrafilter $\fx\in UX$ belongs to the
    closure of $\mathcal{B}\subseteq UX$ whenever
    $\bigcap \mathcal{B}\subseteq \fr$) and anti-monotone with respect to the
    order relation where
    \[
      \fx\le\fy\hspace{1ex}\text{whenever}\hspace{1ex}\forall
      A\in\fx\,.\,\overline{A}\in\fy,
    \]
    for all $\fx,\fy\in UX$. Such maps correspond precisely to subsets
    $\calA\subseteq UX$ which are Zariski closed and down-closed with respect to
    the order relation defined above.

    A pair of $\thU$ -distributors forms an adjunction
    $X\adjunctkmod{\varphi}{\psi}G$ if and only if
    \begin{align*}
      \exists\fx\in UX\,.\, U_{\xi}\varphi(\fx)\;\&\; \psi (\fx)
      &&\text{and}
      &&\forall x\in X\, \forall \fx \in UX\,.\,(\psi (\fx)
         \;\&\; \varphi(x)) \implies \fr\to x.
    \end{align*}
    In terms of the corresponding subsets $A\subseteq X$ and
    $\mathcal{A}\subseteq UX$, these conditions read as
    \begin{align*}
      \exists\fx\in UX\,.\,(A\in\fx \hspace{1ex} \& \hspace{1ex} \fx\in\mathcal{A})
      &&\text{and}
      &&\forall x\in X\, \forall \fx\in UX\,.\,(\fx\in\mathcal{A}\hspace{1ex} \& \hspace{1ex}x\in A)\implies \fx\to x.
    \end{align*}
    From this it follows that $\varphi:G\kmodto X$ is left adjoint if and only
    if of the corresponding closed subset $A\subseteq X$ is
    irreducible. Consequently, a topological space $X$ is Cauchy complete if and
    only if $X$ is weakly sober.
  \item We consider now $\thU=\thU_{\Pp}$, and recall that
    $\thU_{\Pp}\simeq\APP$. Here, a $\thU_{\Pp}$-distributor
    $\varphi :(X,a)\kmodto (Y,b)$ is a $\Pp$-relation $\varphi:UX\relto Y$
    subject to $\varphi\kleisli a\geqslant \varphi$ and
    $b\kleisli \varphi \geqslant\varphi$. These conditions express that, for all
    $y\in Y$, $\fy \in UY$, $\fr \in UX$ and $\fX \in UUX$,
    \begin{align*}
      U_{\xi}a(\fX,\fr)+\varphi(\fr,y)\geqslant \varphi(m_X(\fX),y)
      && \text{and}
      && U_{\xi}\varphi(\fX, \fy)+b(\fy,y)\geqslant \varphi(m_X(\fX),y).
    \end{align*}
    A $\thU_{\Pp}$-distributor of the type $\varphi: G\kmodto X$ can be seen as
    a $\thU_{\Pp}$-functor $\varphi:X\to\Pp$ and it is characterised by
    \[
      U_{\xi}\varphi(\fr)+a(\fr,x)\geq \varphi(x),
    \]
    for $x\in X$ and $\fr \in UX$, and a $\thU_{\Pp}$-distributor of the type
    $\psi: X\kmodto G$ is a mapping $UX\to \Pp$ that satisfies
    \begin{align*}
      U_{\xi}a(\fX,\fr)+\psi(\fr)\geqslant \psi(m_X(\fX))
      && \text{and}
      && U_{\xi}\psi(\fX)\geqslant \psi(m_X(\fX)),
    \end{align*}
    for all $\fr \in UX$ and $\fX \in UUX$. $\thU_{\Pp}$-distributors form an
    adjunction of type $X\adjunctkmod{\varphi}{\psi}G$ if, for all $x\in X$ and
    $\fX \in UUX$,
    \begin{align*}
      0\geqslant \bigwedge_{\fr \in UX} U_\xi \varphi(\fr)+\psi(\fr)
      && \text{and}
      && \psi(m_X(\fX)) + \varphi(x)\geqslant a(m_X(\fX),x).
    \end{align*}
    Furthermore, $\thU_{\Pp}$-distributors type $\varphi: G\kmodto X$ are
    identified with closed variable sets. Here a variable set is a family
    $(A_v)_{v\in [0,\infty]}$ such that, for all $v\in [0,\infty]$,
    $A_v=\bigcap_{u>v}A_u$. Such a variable set is closed whenever, for all
    $u,v \in [0,\infty]$, $\{x\in X \mid d(A_u,x)\leq v \}\subseteq A_{u+v}$,
    where $d(A_u,x)=\inf\{a(\fr, x) \mid \fr \in UA_u\}$. A
    $\thU_{\Pp}$-distributor $\psi:X\kmodto G$ which is right adjoint to
    $\varphi$ is induced by the variable set
    $\mathcal{A}=(\mathcal{A}_v)_{v\in [0,\infty]}$ with
    $\mathcal{A}_v=\{\fr\in UX\mid \forall u\in [0,\infty], \forall x\in A_u,
    a(\fr, x)\leq u+v \}$. Such a variable set $\mathcal{A}$ corresponds to a
    right adjoint of $\varphi$ if and only if $A$ is irreducible, that is, for
    all $u\in [0,\infty]$ with $u>0$, $UA_u \cap \mathcal{A}\neq \varnothing$. A
    $\thU_{\Pp}$-distributor $\varphi : G\kmodto X$ is represented by $x\in X$
    if and only if the induced variable set $A=(A_v)_{v\in [0,\infty]}$ is given
    by $A_v=\{y\in X\mid d(x,y)\leq v\}$ for each $v\in [0,\infty]$. Therefore
    an approach space $X$ is Cauchy complete if and only if each irreducible
    variable set is representable. Finally, this condition is equivalent to $X$
    being weakly sober in the sense of \cite{BLO06}.
  \end{enumerate}
\end{examples}

\subsection{$\thU$-distributors vs $\thU$-functors}

In this subsection we will show that, under suitable conditions, every
$\thU$-category of the form $K(X,a_0,\alpha)$ is Cauchy complete, for
$(X,a_0,\alpha)$ in $(\Cats{\V})^\mU$.  For an $\thU$-category $X=(X,a)$ and
$M\subseteq X$, we define
\[
  \varphi_M(x)=\bigvee_{\fz\in UM}a(\fz,x),
\]
for all $x\in X$. We can view $\varphi_M$ as an $\thU$-relation
$\varphi_M:1\krelto X$ given by $\varphi_M=a\cdot Ui\cdot U!^\circ$ (here
$i:M\hookrightarrow X$ and $!:M\to 1$). It is easy to see that $\varphi_M$ is
actually an $\thU$-distributor $\varphi_M:G\kmodto X$, hence, $\varphi_M:X\to\V$
is an $\thU$-functor. We also not that $\varphi_\varnothing=\bot$ and
$\varphi_{A\cup B}=\varphi_A\vee\varphi_B$.

We import now from \cite[Lemma 3.2 and Corollary 3.3]{HS11}:

\begin{proposition}\label{prop:6}
  Let $\thU=\utheory$ be an ultrafilter theory where $\V$ is completely
  distributive and $\xi:U\V\to\V$ is as in Theorem~\ref{thm:6}. For every
  $\thU$-category $X=(X,a)$, $\fx\in UX$ and $x\in X$,
  $a(\fx,x)=\bigwedge_{A\in\fx}\varphi_A(x)$.
\end{proposition}

Next we analyse left adjoint $\thU$-distributors $\varphi:G\kmodto X$.

\begin{lemma}\label{d:lem:5}
  Let $\thU$ be an ultrafilter theory and $\varphi:G\kmodto X$ be a left adjoint
  $\thU$-distributor with right adjoint $\psi:X\kmodto G$. Then, for every
  $\thU$-distributor $\varphi':G\kmodto X$,
  \[
    [\varphi,\varphi']:=\bigwedge_{x\in
      X}\hom(\varphi(x),\varphi'(x))=\psi\kleisli\varphi'.
  \]
\end{lemma}
\begin{proof}
  Recall that $[\varphi,\varphi']=\varphi'\blackleft\varphi$ is the largest
  element $u\in\V$ with $\varphi(x)\otimes u\le\varphi'(x)$, for all $x\in X$
  (see Subsection~\ref{ssec:V-rel}). From
  \[
    \varphi(x)\otimes\bigvee_{\fx\in UX}\psi(\fx)\otimes \xi U\varphi'(\fx)
    =\bigvee_{\fx\in UX}\varphi(x)\otimes\psi(\fx)\otimes \xi U\varphi'(\fx)
    \le\bigvee_{\fx\in UX} a(\fx,x)\otimes \xi U\varphi'(\fx)=\varphi'(x)
  \]
  we get $\psi\kleisli\varphi'\le [\varphi,\varphi']$. On the other hand, from
  $\varphi\otimes u\le\varphi'$ we get
  \[
    u\le\bigvee_{\fx\in UX}\psi(\fx)\otimes\xi U\varphi(\fx)\otimes u \le
    \bigvee_{\fx\in UX}\psi(\fx)\otimes\xi U(\varphi\otimes u)(\fx) \le
    \bigvee_{\fx\in UX}\psi(\fx)\otimes\xi U\varphi'(\fx).\qedhere
  \]
\end{proof}

\begin{proposition}\label{prop:7}
  Let $\thU=\utheory$ be an ultrafilter theory where $\V$ is completely
  distributive, $\xi$ is as in Theorem~\ref{thm:6}, and $k$ is terminal.
  \begin{enumerate}
  \item For every left adjoint $\thU$-distributor $\varphi:G\kmodto X$,
    \[
      k\le \bigvee_{x\in X}\varphi(x).
    \]
  \item If $k$ is $\vee$-irreducible, then every left adjoint $\thU$-distributor
    $\varphi:G\kmodto X$ is irreducible (that is: $\varphi\neq\bot$ and
    $\varphi\leq \varphi_1\vee\varphi_2$ implies $\varphi\leq \varphi_1$ or
    $\varphi\leq \varphi_2$).
  \end{enumerate}
\end{proposition}
\begin{proof}
  Regarding the first statement, first observe that
  \[
    k\le\bigvee_{\fx\in UX}\psi(\fx)\otimes\xi U\varphi(\fx)\le \bigvee_{\fx\in
      UX}\xi U\varphi(\fx).
  \]
  Let $u\ll k$. Then there is some $\fx\in UX$ with
  $u\le \xi U\varphi(\fx)=\bigwedge_{A\in \fx}\bigvee_{x\in A}\varphi(x)\le
  \bigvee_{x\in X}\varphi(x)$.

  Regarding the second statement, we observe first that $\varphi\neq\bot$ since
  \[
    \bot<k\le \bigvee_{x\in X}\varphi(x).
  \]
  Furthermore, by Lemma~\ref{d:lem:5}, $[\varphi,-]$ preserves finite
  suprema. Therefore, if $\varphi\leq \varphi_1\vee\varphi_2$, then
  \[
    k\le [\varphi,\varphi_1\vee\varphi_2]=[\varphi,\varphi_1]\vee[\varphi,\varphi_2].
  \]
  Since $k$ is $\vee$-irreducible, we conclude that $\varphi\leq \varphi_1$ or
  $\varphi\leq \varphi_2$.
\end{proof}

The following result is inspired by \cite[Lemma III.5.9.1]{HST14} which in turn
is motivated by \cite[Proposition 5.7]{BLO06}

\begin{proposition}\label{prop:8}
  Let $\thU=\utheory$ be an ultrafilter theory where $\V$ is completely
  distributive, $\xi$ is as in Theorem~\ref{thm:6}, and $k$ is terminal and
  approximated. Then every left adjoint $\thU$-distributor $\varphi:G\kmodto X$
  is of the form $\varphi=a(\fx,-)$, for some $\fx\in UX$.
\end{proposition}
\begin{proof}
  First note that from $\{u\in\V\mid u\ll k\}$ is directed it follows that $k$
  is $\vee$-irreducible (see \cite[Remark 4.21]{HR13}).  For every $u\ll k$, put
  $A_u=\{x\in X\mid u\le\varphi(x)\}$. By hypothesis, $A_u\neq\varnothing$. We
  claim that $\varphi\le\varphi_{A_u}$. To see this, put
  $A=\{x\in X\mid\varphi(x)\le \varphi_{A_u}(x)\}$. Since $\varphi_{A_u}(x)=k$
  for every $x\in A_u$, it follows that $A_u\subseteq A$. Put
  $v=\bigvee\{\varphi(x)\mid x\notin A\}$, then $k\not\le v$ since $u\ll v$. By
  construction, $\varphi\le \varphi_{A_u}\vee v$. But $\varphi\le v$ is
  impossible since $k\le\bigvee_{x\in X}\varphi(x)$ and $k\not\le v$, hence
  $\varphi\le\varphi_{A_u}$.

  The directed set $\ff=\{A_u\mid u\ll k\}$ is disjoint from the ideal
  $\fj=\{B\subseteq X\mid \varphi\not\le\varphi_B\}$, hence there is some
  ultrafilter $\fx\in UX$ with $\ff\subseteq\fx$ and
  $\fx\cap\fj=\varnothing$. Therefore
  \[
    \varphi\le\bigwedge_{A\in\fx}\varphi_A=a(\fx,-)
  \]
  and
  \[
    \varphi(x)\ge a(\fx,x)\otimes \xi U\varphi(\fx)\ge a(\fx,x),
  \]
  for all $x\in X$.
\end{proof}

\begin{corollary}\label{d:cor:1}
  Under the conditions of Proposition~\ref{prop:8}, every $\thU$-category in the
  image of
  \[
    K:(\Cats{\V})^\mU \longrightarrow\Cats{\thU}
  \]
  is Cauchy complete. In particular, the $\thU$-category $\V$ is Cauchy
  complete.
\end{corollary}
\begin{proof}
  Given a left adjoint $\thU$-distributor $\varphi:G\kmodto X$, we have
  $\varphi=a(\fx,-)=a_0(\alpha(\fx),-)$.
\end{proof}

For our next result, we recall that the forgetful functor
$(-)_0:\Cats{\thU}\to\Cats{\V}$ has a left adjoint $F:\Cats{\V}\to\Cats{\thU}$
sending a $\V$-category $(X,a_0)$ to the $\thU$-category
$(X,e_X^\circ\cdot\Uxi a_0)$, and leaving maps unchanged.

\begin{proposition}
  Let $\thU$ be an ultrafilter theory. Then the following assertions hold.
  \begin{enumerate}
  \item $F$ sends fully faithful $\V$-functors to fully faithful
    $\thU$-functors.
  \item If $\thU$ is strict, then $F$ sends fully dense $\V$-functors to fully
    dense $\thU$-functors.
  \end{enumerate}
\end{proposition}
\begin{proof}
  For a $\V$-functor $f:(X,a_0)\to(Y,b_0)$, we write
  \begin{align*}
    a=e_X^\circ\cdot\Uxi a_0 &&\text{and}&& b=e_Y^\circ\cdot\Uxi b_0
  \end{align*}
  for the corresponding $\thU$-structures. Assume first that
  $f:(X,a_0)\to(Y,b_0)$ is fully faithful. Then
  \[
    f^\Tast\kleisli f_\Tast =f^\circ \cdot e_Y^\circ\cdot\Uxi b_0\cdot Uf
    =e_X^\circ\cdot Uf^\circ\cdot\Uxi b_0\cdot Uf =e_X^\circ\cdot
    \Uxi(f^\circ\cdot b_0\cdot f) =a
  \]
  Assume now that $\thU$ is strict and $f$ is fully dense. Now we calculate:
  \begin{multline*}
    f_\Tast\kleisli f^\Tast =b\cdot Uf\cdot Uf^\circ\cdot \Uxi b\cdot m_Y^\circ
    =e_Y^\circ\cdot\Uxi b_0\cdot Uf\cdot Uf^\circ\cdot Ue_Y^\circ\cdot\Uxi\Uxi b_0\cdot m_Y^\circ\\
    =e_Y^\circ\cdot\Uxi b_0\cdot Uf\cdot Uf^\circ\cdot Ue_Y^\circ\cdot
    m_Y^\circ\cdot\Uxi b_0 =e_Y^\circ\cdot\Uxi b_0\cdot Uf\cdot Uf^\circ\cdot
    \Uxi b_0=b\qedhere
  \end{multline*}
\end{proof}

\begin{theorem}
  Let $\thU$ be a strict ultrafilter theory. Then
  $(-)_0:\Cats{\thU}\to\Cats{\V}$ sends Cauchy complete $\thU$-categories to
  Cauchy complete $\V$-categories.
\end{theorem}
\begin{proof}
  Just note that a $\V$-category (resp.~$\thU$-category) is Cauchy complete if
  and only if it is injective with respect to fully faithful and fully dense
  $\V$-functors (resp.~$\thU$-functors) as it was proven in \cite[Theorems~3.10
  and 5.11]{HT10}.
\end{proof}

\begin{corollary}\label{d:cor:3}
  Let $\thU=\utheory$ be a strict ultrafilter theory where $\V$ is completely
  distributive, $\xi$ is as in Theorem~\ref{thm:6}, and $k$ is terminal and
  approximated. Then, for every $(X,a_0,\alpha)$ in $(\Cats{\V})^\mU$, the
  $\V$-category $(X,a_0)$ is Cauchy complete. In particular, every compact
  separated $\V$-category is Cauchy complete.
\end{corollary}

For $\thU=\thU_\two$, the result above is vacuous since every ordered set is
Cauchy complete. As we already pointed out in Section~\ref{sec:introduction}, a
stronger result holds in this case: the underlying order of a sober space is
codirected complete. In the next subsection we proof a similar result for
$\thU$-categories, under additional conditions on the quantale $\V$.

\begin{remark}
  A related study of properties of metric spaces via approach spaces can be
  found in \cite{LZ16}. Among other results, it is shown there that in the
  underlying metric of an approach space every \emph{forward Cauchy sequence}
  converges (see \cite{BBR98,Wag97}). We will come back to this notion in the
  next subsection.
\end{remark}

\subsection{Codirected complete $\V$-categories}
\label{sec:codir-compl-v}

In this subsection we look at Cauchy completeness of $\V$-categories from a
different perspective, namely as (co)completeness with respect to some choice of
(co)limit weights. In this paper we need only very particular limits and
colimits, therefore we refer for more information to \cite{KS05,Stu14} and
recall here only what we believe is essential for our paper.

As the starting point, we assume that a saturated class $\Phi$ of limit weights
$\varphi:G\modto X$ is given; examples of such choice are given below. For each
$\V$-category $X$, we write $\Phi(X)$ to denote the weights with codomain
$X$. Moreover, we consider $\Dists{\V}(G,X)$ as a $\V$-subcategory of
$\Rels{\V}(1,X)\simeq \V^X$ and $\Phi(X)$ as a $\V$-subcategory of $\Dists{\V}(G,X)^\op$,
this way the mapping
\[
  \coyoneda_X^\Phi:X\longrightarrow\Phi(X),\,x\longmapsto x_*
\]
is a $\V$-functor. A $\V$-category $X$ is called \df{$\Phi$-complete} whenever
$\coyoneda_X^\Phi$ has a right adjoint
\[
  {\inf}^\Phi_X:\Phi(X) \longrightarrow X.
\]
Intuitively, ${\inf}^\Phi_X$ calculates the infimum of a limit weight
$\varphi:G\modto X$. The assumption that $\Phi$ is saturated guarantees that
each $\Phi(X)$ is $\Phi$-complete; in fact, it is the free $\Phi$-completion of
$X$. Dually, notions of cocompleteness depend on a choice of a saturated class
$\Psi$ of colimit weights $\psi:X\modto G$. Then a $\V$-category $X$ is
$\Psi$-cocomplete if and only if the $\V$-functor
\[
  X\longrightarrow\Psi(X),\,x \longmapsto x^*
\]
has a left adjoint. Here we consider $\Psi(X)$ as a $\V$-subcategory of
$\Dists{\V}(X,G)$.

\begin{remark}
  For a saturated class $\Phi$ of limit weights, a $\V$-category $X$ is
  $\Phi$-complete if and only if there exists a $\V$-functor $I:\Phi(X)\to X$
  with $I\coyoneda_X^\Phi\simeq 1_X$; such a $\V$-functor is necessarily right
  adjoint to $\coyoneda_X^\Phi$.
\end{remark}

For instance,
\[
  \Phi=\{\text{all left adjoint $\V$-distributors $\varphi:G\modto X$ with
    domain $G$}\}
\]
is a saturated class of limit weights, and a $\V$-category $X$ is
$\Phi$-complete if and only if $X$ is Cauchy complete. The following definition
provides another important example of a saturated class of limit weights.

\begin{definition}\label{d:def:1}
  Let $\V$ be a quantale. A $\V$-distributor $\varphi_0:G\modto X$ with domain
  $G$ is called \df{codirected} whenever the $\V$-functor
  \[
    [\varphi_0,-]:\Dists{\V}(G,X)\longrightarrow\V
  \]
  preserves finite suprema and tensors; that is, for all
  $\varphi,\varphi':G\modto X$ and all $u\in\V$,
  \begin{align*}
    [\varphi_0,\bot]&=\bot,
    & [\varphi_0,\varphi\vee\varphi']&=[\varphi_0,\varphi]\vee[\varphi_0,\varphi']
    &  [\varphi_0,u\otimes\varphi]&=u\otimes[\varphi_0,\varphi].
  \end{align*}
\end{definition}

We note that the class $\DirLimCls$ of all codirected $\V$-distributors
$\varphi:G\modto X$ is saturated (see \cite{KS05}).

\begin{definition}\label{d:def:2}
  A $\V$-category $X$ is called \df{codirected complete} whenever $X$ is
  $\DirLimCls$-complete.
\end{definition}

For a left adjoint $\V$-distributor $\varphi:G\modto X$ with right adjoint
$\psi:X\modto G$, we have
\[
  [\varphi,-]=\psi\cdot-
\]
since $\varphi\cdot-\dashv\psi\cdot-$ and $\varphi\cdot-\dashv[\varphi,-]$;
which shows that $\varphi:G\modto X$ is codirected. Therefore every codirected
complete $\V$-category is Cauchy complete.

\begin{example}
  For $\V=\two$, we can interpret every $\two$-distributor $\varphi:G\modto X$
  as an upclosed subset $A\subseteq X$ of $X$. Then $A$ is codirected in the
  sense of Definition~\ref{d:def:1} if and only if $A$ is codirected in the
  usual sense; that is, $A\neq\varnothing$ and, for all $x,y\in A$, there is
  some $z\in A$ with $z\leq x$ and $z\leq y$.
\end{example}

We recall now that, by Theorem~\ref{thm:3}, $\thU$-distributors of type
$G\kmodto X$ correspond to $\thU$-functors $X\to\V$; and with this perspective
we can consider $\Dists{\thU}(G,X)$ as a $\V$-subcategory of
$\Dists{\V}(G,X_0)$.

\begin{proposition}\label{d:prop:2}
  For every ultrafilter theory $\thU=\utheory$, the inclusion $\V$-functor
  \[
    \Dists{\thU}(G,X) \longrightarrow \Dists{\V}(G,X_0)
  \]
  has a left adjoint
  \[
    \overline{(-)}:\Dists{\V}(G,X_0)\longrightarrow\Dists{\thU}(G,X).
  \]
  Moreover, if $\thU$ is pointwise strict and compatible with finite suprema,
  then $\Dists{\thU}(G,X)$ is closed in $\Dists{\V}(G,X_0)$ under finite suprema
  and tensors.
\end{proposition}
\begin{proof}
  By \cite[Corollary~5.3]{Hof07}, the $\V$-category $\Dists{\thU}(G,X)$ is
  closed in $\Dists{\V}(G,X_0)$ under weighted limits. The additional conditions
  guarantee that the maps
  \begin{align*}
    t_u:\V\to\V &&\text{and}&& \vee:\V\times\V\to\V 
  \end{align*}
  are $\thU$-functors, for every $u\in\V$; which justifies the second claim.
\end{proof}

\begin{corollary}\label{d:cor:2}
  Let $\thU=\utheory$ be a strict ultrafilter theory compatible with finite
  suprema so that $k$ is terminal and approximated. Then, for every codirected
  $\V$-distributor $\varphi:G\modto X$, the $\thU$-distributor
  $\overline{\varphi}:G\kmodto X$ is left adjoint in $\Dists{\thU}$.
\end{corollary}
\begin{proof}
  We recall first from Proposition~\ref{d:prop:4} that, under these assumptions,
  $k$ is $\vee$-irreducible. Using the adjunction of Proposition~\ref{d:prop:2},
  the $\V$-functor
  \[
    [\overline{\varphi},-]:\Dists{\thU}(G,X)\longrightarrow\V
  \]
  is equal to the composite
  \[
    \Dists{\thU}(G,X)\longrightarrow\Dists{\V}(G,X_0)
    \xrightarrow{\,[\varphi,-]\,}\V,
  \]
  and therefore $[\overline{\varphi},-]$ preserves tensors and finite
  suprema. By \cite[Propositions~2.15 and 3.5]{HS11},
  $\overline{\varphi}:G\kmodto X$ is left adjoint in $\Cats{\thU}$. Note that
  the notation regarding distributors in \cite{HS11} is dual to ours.
\end{proof}

\begin{theorem}
  Let $\thU=\utheory$ be a strict ultrafilter theory compatible with finite
  suprema where $\V$ is completely distributive, $\xi$ is as in
  Theorem~\ref{thm:6}, and $k$ is terminal and approximated. Then, for every
  $\V$-categorical compact Hausdorff space $(X,a_0,\alpha)$, the $\V$-category
  $(X,a_0)$ is codirected complete.
\end{theorem}
\begin{proof}
  Let $\varphi:G\modto X$ be a codirected $\V$-distributor. By
  Corollaries~\ref{d:cor:1} and \ref{d:cor:2}, there is some $y\in X$ with
  $\overline{\varphi}=y_\Tast=y_*$. Then, for every $x\in X$,
  \[
    [\varphi,x_*]=[\overline{\varphi},x_*]=[y_*,x_*]=a_0(x,y).
  \]
  This proves that $y_X:X\to\DirLimCls(X)$ has a right adjoint in $\Cats{\V}$.
\end{proof}

We finish this subsection by exhibiting a connection with other accounts of
``codirected complete metric spaces'' which appear in the literature. Firstly,
non-symmetric versions of Cauchy sequences and their limits are introduced in
\cite{Smy88} and further studied in \cite{Rut96,BBR98}: a sequence
$(x_n)_{n\in\N}$ in a metric space $(X,d)$ is called \df{forward-Cauchy}
whenever
\[
  \forall\varepsilon>0\,\exists N\in\N\,\forall n\geq m\geq
  N\,.\,d(x_m,x_n)<\varepsilon,
\]
and $(x_n)_{n\in\N}$ is called \df{backward-Cauchy} whenever
\[
  \forall\varepsilon>0\,\exists N\in\N\,\forall n\geq m\geq
  N\,.\,d(x_n,x_m)<\varepsilon.
\]
The definitions above extend naturally to nets (see \cite{FSW96}), and in
\cite{Vic05} it is shown that that forward-Cauchy nets in metric spaces
correspond precisely to those $\Pp$-distributors $\psi:X\modto G$ with the
property that the $\V$-functor
\[
  \psi\cdot-:\Dists{\Pp}(G,X)\longrightarrow\Pp,\,\varphi
  \longmapsto\psi\cdot\varphi
\]
preserves finite meets. On the other hand, in \cite{HW12} it is shown that these
distributors do \emph{not} coincide with forward-Cauchy nets for $\V=\Pm$. Such
$\Pp$-distributors are called \emph{flat} in \cite{Vic05}; however, in this
paper we deviate slightly from the notation of \cite{Vic05}.

\begin{definition}
  A $\V$-distributor $\psi:X\modto G$ is called \df{flat} whenever
  $\psi\cdot-:\Dists{\V}(G,X)\to\V$ preserves finite infima and cotensors.
\end{definition}

In order to compare these two notions of ``directedness'', we restrict our study
to a certain type of quantales.

\begin{definition}
  We call a quantale $\V=\quantale$ a \df{Girard quantale} whenever $\V$ has a
  dualising element $D\in\V$; that is, for every $u\in \V$,
  $u=\hom(\hom(u,D),D)$.
\end{definition}

This type of quantales is introduced in \cite{Yet90}, we also refer to
\cite{Was09} for a study of categories enriched in a Girard quantale.

\begin{examples}
  The quantale $\two=\{0,1\}$ and the quantale $[0,1]$ with the \L{}ukasiewicz
  tensor $\otimes=\odot$ are Girard quantales, with dualising object the bottom
  element $0$.
\end{examples}

For $\V=\quantale$ being a Girard quantale with dualising element $D$, we write
$u^\bot=\hom(u,D)$. As shown in \cite{Yet90}, the operations $(-)^\bot$ and
$\otimes$ allow us to determine the internal hom of $\V$: for all $u,v\in\V$,
\[
  \hom(u,v)=(u\otimes v^\bot)^\bot.
\]

\begin{lemma}\label{d:lem:1}
  The map $(-)^\bot:\V\to\V^\op$ is a $\V$-functor. Hence, $\V\simeq\V^\op$ in
  $\Cats{\V}$.
\end{lemma}
\begin{proof}
  For all $u,v\in\V$ we have
  \[
    \hom(u,v)\otimes\hom(v,D)\le\hom(u,D),
  \]
  which is equivalent to $\hom(u,v)\leq\hom(v^\bot,u^\bot)$.
\end{proof}

Hence, for every $\varphi:G\modto X$ in $\Dists{\V}$,
$\varphi^\bot(x)=\varphi(x)^\bot$ defines a $\V$-distributor
$\varphi^\bot:X\modto G$. Hence, the isomorphism of Lemma~\ref{d:lem:1} induces
a $\V$-isomorphism
\[
  (-)^\bot:\Dists{\V}(G,X) \longrightarrow\Dists{\V}(X,G)^\op.
\]

\begin{proposition}
  Let $\V=\quantale$ be a Girard quantale, $X$ a $\V$-category and
  $\varphi_0:G\modto X$ in $\Dists{\V}$. Then the diagram
  \[
    \xymatrix{\Dists{\V}(G,X)\ar[r]^{(-)^\bot}\ar[d]_{[\varphi_0,-]} &
      \Dists{\V}(X,G)^\op\ar[d]^{(-\cdot\varphi_0)^\op} \\
      \V\ar[r]_{(-)^\bot} & \V^\op}
  \]
  commutes.
\end{proposition}
\begin{proof}
  Let $\varphi:G\modto X$ be a $\V$-distributor. Then
  \begin{align*}
    [\varphi_0,\varphi]^\bot
    &=\left(\bigwedge_{x\in X}\hom(\varphi_0(x),\varphi(x))\right)^\bot\\
    &=\bigvee_{x\in X}\hom(\varphi_0(x),\varphi(x))^\bot\\
    &=\bigvee_{x\in X}\varphi_0(x)\otimes\varphi(x)^\bot.\qedhere
  \end{align*}
\end{proof}

\begin{corollary}
  Let $\V=\quantale$ be a Girard quantale. Then a $\V$-distributor
  $\varphi \colon G\modto X$ is codirected if and only if the $\V$-functor
  $-\cdot\varphi:\Dists{\V}(X,G)\to\V$ preserves cotensors and finite
  infima. Hence, $X$ is codirected complete if and only if $X^\op$ is cocomplete
  with respect to all flat $\V$-distributors $\psi:X\modto G$.
\end{corollary}

\newcommand{\etalchar}[1]{$^{#1}$}



\end{document}